%% file: BRUS.tex
\documentclass%[12pt]
{amsproc}

\usepackage{amsmath}
\usepackage{amsthm}
\usepackage{amssymb}
\usepackage{graphicx}
\usepackage{multirow}

\DeclareGraphicsExtensions{.png,.pdf,.eps}
\usepackage[all,2cell]{xy}
\usepackage{tikz}
\usepackage{pgf}
\usepackage{enumerate}
\usepackage{mathdots}
\usetikzlibrary{cd}
\usepackage{caption}
\usepackage{adjustbox}
\usetikzlibrary{matrix,arrows,backgrounds}

\usepackage{array}
\newcolumntype{H}{>{\setbox0=\hbox\bgroup}c<{\egroup}@{}}
\usepackage{hyperref}

\usepackage[customcolors]{hf-tikz}

\newtheorem{theorem}{Theorem}[section]
\newtheorem{lemma}[theorem]{Lemma}
\newtheorem{corollary}[theorem]{Corollary}
\newtheorem{proposition}[theorem]{Proposition}

\newtheorem{step2}{Step} % ADD ONE FOR EVERY THEOREM THAT WE WANT TO DIVIDE INTO STEPS 

\theoremstyle{definition}

\newtheorem{definition}[theorem]{Definition}

\newtheorem{setup}[theorem]{Set-up}

\newtheorem{remark}[theorem]{Remark}

\newtheorem{set}[theorem]{Setup}
\newtheorem{example}[theorem]{Example}

\title[Mori Dream Bonds and $\C^*$-actions]{Mori Dream Bonds and $\C^*$-actions}

\author[Barban]{Lorenzo Barban}
\address{Dipartimento di Matematica, Universit\`a degli Studi di Trento, via
	Sommarive 14 I-38123 Povo di Trento (TN), Italy}
\email{lorenzo.barban@unitn.it}

\author[Romano]{Eleonora A. Romano}
\address{Dipartimento di Matematica, Universit\`a degli Studi di Genova, via Dodecaneso 35, I-16146, Genova (GE), Italy}
\email{eleonoraanna.romano@unige.it}

\author[Sol\'a Conde]{Luis E. Sol\'a Conde}
\address{Dipartimento di Matematica, Universit\`a degli Studi di Trento, via
Sommarive 14 I-38123 Povo di Trento (TN), Italy}
\email{eduardo.solaconde@unitn.it}

\author[Urbinati]{Stefano Urbinati}
\address{Universit\`a degli Studi di Udine, Dipartimento di Scienze Matematiche,
Informatiche e Fisiche, Via delle Scienze, 206 - 33100 Udine, Italy}
\email{stefano.urbinati@uniud.it}


\thanks{}

\input{definitions.tex}
\begin{document}
\begin{abstract}
	We construct a correspondence between Mori dream regions arising from small modifications of normal projective varieties and $\C^*$-actions on polarized pairs which are bordisms. Moreover, we show that the Mori dream regions constructed in this way admit a chamber decomposition on which the models are the geometric quotients of the $\C^*$-action. In addition we construct, from a given $\C^*$-action on a polarized pair for which there exist at least two admissible geometric quotients, a $\C^*$-equivariantly birational $\C^*$-variety, whose induced action is a bordism, called the pruning of the variety.
\end{abstract}
\maketitle
\tableofcontents

%%%%%%%%%%%%%%%%% inputs

\input{intro}

\input{prelim}

\input{constructionbordism}

\input{MDR2}

\input{bordismtoMDR}

\input{example}

\bibliographystyle{plain}
\bibliography{bibliomin}
\end{document}

%% file: definitions.tex
\usepackage[textsize=tiny]{todonotes}

\usepackage{enumitem}
\newcommand\ignore[1]{}

\def\Div{\operatorname{Div}}
\def\div{\operatorname{div}}
\def\CDiv{\operatorname{CDiv}}
\DeclareMathOperator{\HH}{H}
\DeclareMathOperator{\Sym}{Sym}

\DeclareMathOperator{\Cox}{Cox}

\def\C{{\mathbb C}}

\def\N{{\mathbb N}}
\def\P{{\mathbb P}}
\def\Q{{\mathbb Q}}
\def\R{{\mathbb R}}
\def\Z{{\mathbb Z}}

\def\cC{{\mathcal C}}

\def\cN{{\mathcal{N}}}
\def\cO{{\mathcal{O}}}

\def\cS{{\mathcal S}}

\def\cY{{\mathcal Y}}
\def\Q{{\mathbb{Q}}}

\def\operatorname#1{\mathop{\rm #1}\nolimits}

\def\Proj{\operatorname{Proj}}

\def\Hom{\operatorname{Hom}}

\def\Hom{\operatorname{Hom}}

\def\Spec{\operatorname{Spec}}

\def\codim{\operatorname{codim}}

\def\NU{{\operatorname{N^1}}}

\def\GX{\mathcal{G}\!X}

\def\GZ{\mathcal{G}\!Z}

\newcommand{\st}{\operatorname{s}}

\newcommand{\pb}{\ar@{}[dr]|{\text{\pigpenfont J}}}
\def\ol{\overline}
\def\tl{\widetilde}

\makeatletter
\newcommand{\xleftrightarrow}[2][]{\ext@arrow 3359\leftrightarrowfill@{#1}{#2}}
\makeatother
\newcommand{\xdasharrow}[2][->]{
\tikz[baseline=-\the\dimexpr\fontdimen22\textfont2\relax]{
\node[anchor=south,font=\scriptsize, inner ysep=1.5pt,outer xsep=2.2pt](x){#2};
\draw[shorten <=3.4pt,shorten >=3.4pt,dashed,#1](x.south west)--(x.south east);
}}
\newcommand{\hooklongrightarrow}{\lhook\joinrel\longrightarrow}

\def\Mo{\operatorname{\hspace{0cm}M}}

\def\We{\operatorname{\hspace{0cm}W}}

\newcommand{\git}{\mathbin{/\mkern-6mu/}}

%% file: intro.tex
%!TEX root = BRUS.tex

\section{Introduction}\label{sec:intro}

The existence of an interplay between birational geometry and geometric invariant theory is a well known fact since the 1980's. In his work \cite{ReidFlip} Reid explains how flips --fundamental birational transformations in modern birational geometry-- could be understood in terms of variation of GIT (see also \cite{hu}). On the other hand Thaddeus \cite{Thaddeus1996} showed that the set of all possible linearizations of a torus action has a natural structure of convex cone in the Neron-Severi space of the variety, divided in (convex conic) chambers on whose interior the associated GIT quotient is invariant, and described the birational transformations associated to wall-crossing in this chamber decomposition. The  construction
of the wall-crossing transformations in terms of weighted blow-ups has also appeared in \cite{brionprocesi}. Moreover, W{\l}odarczyk (see \cite{Wlodarczyk}) used these ideas to address the Weak factorization conjecture of birational transformations by means of the concept of cobordism. Finally, these ideas gave rise to the concept of Mori dream space (MDS for short), introduced by Hu and Keel (cf. \cite{HuKeel}). In a nutshell, the birational geometry of an MDS $X$ is governed by the different GIT quotients of an affine variety $\Spec(\Cox(X))$ by the action of a torus.

Recently the abovementioned relation was exploited in the classification of certain complex projective varieties admitting a $\C^*$-action of a %particularly 
simple type (cf. \cite{WORS1}), which then was used in \cite{WORS2} to prove some partial version of the LeBrun-Salamon conjecture. In the paper \cite{WORS1} the authors studied $\C^*$-action counterparts of some birational transformations, such as Atiyah flips and special Cremona transformations. These results were then generalized in \cite{WORS3} to study equivariant birational modifications of smooth projective varieties with a $\C^*$-action under some %--quite restrictive-- 
hypotheses. The birational geometry of these varieties is then completely determined in a precise manner by the corresponding GIT-quotients of the $\C^*$-action (see for instance \cite[Theorem 1.1]{WORS3}). The idea is the following: given a polarized pair $(X,L)$, where $X$ is a normal projective variety and $L$ is an ample line bundle on $X$, with a $\C^*$-action satisfying some hypotheses (see \cite[\S 4]{WORS3}), one gets a naturally ordered chain of birational transformations among the geometric quotients $\GX(i,i+1)$ of $X$:
$$
\xymatrix{\GX(0,1)\ar@{-->}[r]^{\psi_1}&\GX(1,2)\ar@{-->}[r]^(.50){\psi_2}&\dots\ar@{-->}[r]^(.30){\psi_{r-1}}&\GX(r-1,r),}
$$
which are flips (see Example \ref{ex:grass}, \cite[Theorem 3.1]{WORS3}). 
Moreover, given two of these geometric quotients $\GX(i,i+1)$, $\GX(j,j+1)$, $i\leq j$, one may find a variety $X'$ -- $\C^*$-equivariantly -- birationally equivalent to $X$, admitting a $\C^*$-action whose geometric quotients are $\GX(k,k+1)$, $i\leq k\leq j$, and one may obtain $X$ from $X'$ by means of a blow-up and a sequence of flips (\cite[Theorem 1.1]{WORS3}). 

With this in mind, a natural question arises: given a birational map $\phi:Y_-\dashrightarrow Y_+$ among normal projective varieties, is it  possible to find a variety $X$, endowed with a $\C^*$-action such that $Y_\pm$ are the extremal geometric quotients of $X$ and such that $\phi$ coincides with the map induced by the $\C^*$-action? If this is the case, we say that $X$ is a \emph{geometric realization} of the birational map $\phi$. If such a variety $X$ exists, we get a natural decomposition of $\phi$ in terms of simpler birational transformations, %(essentially flips, blowups and blowdowns), 
that, moreover, can be read out of the birational geometry of $X$. 

In this paper we address this problem, extending the results of \cite{WORS3} to broader generality. We consider a particular type of polarized pairs $(X,L)$ admitting a $\C^*$-action, called {\em bordisms}, and study how their birational geometry is ruled by the induced birational transformation between the corresponding polarized extremal fixed point components $(Y_\pm,L_\pm)$. A bordism is a projective variety admitting a $\C^*$-action whose only invariant divisors are its extremal fixed point components, which are requested to be geometric quotients of $X$. Although at a first glimpse bordisms may seem to be very particular cases of torus actions, we show that they can be constructed upon many examples of torus actions by means of a birational transformation called {\em pruning} (see Definition \ref{def:pruning}). Different prunings of the same torus action give rise to $\C^*$-equivariantly birationally equivalent varieties (see Theorem \ref{theorem:Xbordism}). A central definition in this work is the one of {\em Mori Dream Bond} ({\em MDB} for short):

\begin{definition}\label{def:MDRtypeintro}
	Let $Y_-$ be a normal projective variety, 
	$L_-,L_+\in\CDiv(Y_-)$ be two effective Cartier divisors such that $L_-$ is ample, $R(Y_-;L_+)$ is finitely generated and the natural map $\phi:Y_-\dashrightarrow Y_+:=\Proj(R(Y_-;L_+))$ is a small modification.
	\\We say that $(L_-,L_+)$ is a {\em Mori Dream Bond} on $Y_-$ if the cone $\cC\subset \CDiv(Y_-)_{\Q}$ generated by $L_-,L_+$ is a Mori Dream Region, that is the multisection ring
	$$R(Y_-; L_-,L_+)=\bigoplus_{a,b\geq 0} \HH^0(Y_-,aL_-+bL_+)$$
	is a finitely generated $\C$-algebra.
\end{definition}

This condition is essentially a local version of the MDS definition. Our main results show that essentially bordisms are torus action counterparts of small modifications defined by an MDB. More concretely we prove the following:

\begin{theorem}\label{thm:MDR2intro}
	Let $\phi:Y_-\dashrightarrow Y_+$ be a small modification defined by an MDB $(L_-,L_+)$ on a normal projective variety $Y_-$. Then there exists a normal projective variety $X$ admitting a $\C^*$-action which is a bordism and realizes geometrically the map $\phi$.
\end{theorem}

Let us recall that a $\C^*$-action is {\em equalized at the sink $Y_-$ and source $Y_+$} if every point, whose orbit converges either to $Y_-$ or to $Y_+$, has trivial isotropy group (cf. \S \ref{ssec:BB}). Moreover we define the {\em bandwidth} of the $\C^*$-action on $(X,L)$ as the difference $\delta=\mu_L(Y_+)-\mu_L(Y_-)$, where by $\mu_L(Y_{\pm})$ we denote the weight of the induced $\C^*$-action on any fiber of $L$ over a point of $Y_{\pm}$ (cf. \S \ref{ssec:BB}).

\begin{theorem}\label{theorem:mdrintro}
	Let $X$ be a normal $\Q$-factorial projective variety, let $(X,L)$ be a polarized pair endowed with a $\C^*$-action of bandwidth $\delta$ which is a bordism and is equalized at the sink $Y_-$ and the source $Y_+$, and let $\psi:Y_-\dashrightarrow Y_+$ be the induced birational map. Denote by $L_-=L_{|Y_-}$, and  by $L_+=L_{\mid Y_-}-\delta Y_{-\mid Y_-}$. Then $(L_-,L_+)$ is an MDB, whose associated map is $\psi$.
\end{theorem}

Moreover, in the spirit of \cite{KKL}, we show the existence of a chamber decomposition of an MDB; we refer to Definition \ref{definition:chamber} for the notion of chamber model.

\begin{theorem}\label{theorem:decompositionintro}
	In the situation of Theorem \ref{theorem:mdrintro}, the rational polyhedral cone $\cC=\langle L_-,L_+ \rangle\subset \CDiv(Y_-)_{\Q}$  admits a subdivision $\cC=\bigcup_{i=0}^{r-1}\cC_i$, where every $\cC_i=\langle L_{|Y_-}-a_iY_{-|Y_-},L_{|Y_-}-a_{i+1}Y_{-|Y_-} \rangle$ is a chamber and $a_0,\ldots,a_r$ are the critical values of the $\C^*$-action on $(X,L)$. Moreover, for every $i$ the chamber model of $\cC_i$ is $\GX(i,i+1)$.
\end{theorem}

We conclude by constructing a natural geometric realization of the standard cubo-cubic Cremona involution of $\P^3$. We then study its possible prunings, and realize one of them as a geometric realization of an MDB constructed upon the blow-up of $\P^3$ along the $4$ coordinates points.

\noindent\medskip\par{\bf Outline.} After introducing some background material about $\C^*$-actions, bordisms, and MDBs (Section \ref{sec:prelim}), we describe the pruning construction in Section \ref{sec:constructionbordism}%, where we prove Theorem \ref{theorem:Xbordism}
. Section \ref{sec:MDR2} is devoted to the proof of Theorem \ref{thm:MDR2intro} (cf. Theorem \ref{thm:MDR2}); in Section \ref{sec:bordismtoMDR} we prove Theorems \ref{theorem:mdrintro} and \ref{theorem:decompositionintro} (see Theorems \ref{theorem:mdr} and \ref{theorem:decomposition} respectively). In Section \ref{sec:Cremona} we study an explicit example of geometric realization obtained by the factorization of the standard Cremona involution of $\P^3$.

\noindent\medskip\par{\bf Acknowledgements.} We would like thank Jarek Wi\'sniewski for helpful discussions at the early stages of this work. We also thank Alex K{\"u}ronya, Friedrich Knop and Michel Brion for fruitful conversations. This project was also carried out during two stays of the first author both at the University of Genova and at the University of Udine; he would like to thank the Departments of Mathematics for the kind hospitality and helpful discussions.

%% file: prelim.tex
%!TEX root = BRUS.tex

\section{Preliminaries}\label{sec:prelim}

\subsection{Notation}

We work over the field $\C$ of complex numbers. Given a normal projective variety $Y$, we denote by $\Div(Y)$ (resp. $\CDiv(Y)$) the group of Weil (resp. Cartier) divisors on $Y$, and by $\Div(Y)_{\Q}$ (resp. $\CDiv(Y)_{\Q}$) the associated vector space with rational coefficients. 
A \emph{small modification} is a birational map $\phi:Y\dashrightarrow Y'$ among normal projective varieties which is an isomorphism in codimension 1. Abusing notation, the strict transform of a Weil divisor $D$ in $Y'$ will still be denoted by $D$. If in addition $Y$ and $Y'$ are $\Q$-factorial we call $\phi$ a \emph{small $\Q$-factorial modification} (SQM for short). In this case the strict transform via $\phi$ allows us to identify $\CDiv(Y)_{\Q}$ with $\CDiv(Y')_{\Q}$.  

By a \emph{polarized pair} we mean a pair $(Y,L)$, where $Y$ is a normal projective variety and $L$ is an ample line bundle. 

\subsection{Mori dream regions and bonds}\label{ssec:MDR}

Given a finite set of effective Cartier divisors $L_i$, for $i=1,\dots,k$, on a normal projective variety $Y$, we define the \emph{section ring} $R(Y;L_1,\dots,L_k)$ as the multigraded $\C$-algebra
$$
R(Y;L_1,\dots,L_k):=\bigoplus_{m_1,\ldots,m_k\in \N} \HH^0(Y,\cO_Y(m_1L_1+\ldots+m_kL_k)).
$$
Note that the product in $R(Y;L_1,\dots,L_k)$ is induced by the inclusions of each $\HH^0(Y,\cO_Y(m_1L_1+\ldots+m_kL_k))$ into the field of rational functions of $Y$, and so it depends on the particular choice of the effective Cartier divisors $L_i$ (not only on their linear equivalence class, cf. \cite[Remark, p. 341]{HuKeel}). 

In the case $k=1$, the finite generation of section rings is a classical problem in birational geometry. If $\{L_1,\dots, L_k\}$ is a set of generators of the effective cone of $Y$ (in the case such a set exists), the finite generation of the section ring --which is then called the \emph{Cox ring} of $Y$ -- is equivalent to the birational geometric concept of Mori dream space. In the spirit of \cite[Theorem 2.13]{HuKeel} we introduce the following:

\begin{definition}
	Let $Y$ be a normal projective variety, and let $\cC=\langle L_1,\ldots, L_k\rangle$ be a rational polyhedral cone in $\CDiv(Y)_{\Q}$, where $L_i$ are effective. The cone $\cC$ is a \emph{Mori dream region} --MDR for short-- if the multisection ring 
	$R(Y;L_1,\ldots,L_k)$ is a finitely generated $\C$-algebra.
\end{definition} 

\begin{remark}\label{remark:dicussionMDR}
	Note that the meaning of the term Mori Dream Region is not unique in the literature. The one we give and will need for this work is essentially similar to the one introduced in \cite[Theorem 2.13]{HuKeel}. We refer to \cite[\S 9.2]{okawa}, \cite[Theorem 4.2]{KKL} and \cite[\S 5]{KU} for a complete picture of the different properties required for MDRs. 
	%In particular, the definition given in \cite{HuKeel} just considers a subcone of the effective cone whose associated divisorial ring is finitely generated. This property itself is enough to achieve a \lq\lq good\rq\rq chamber decomposition of the cone, that is sufficient when considering $\mathbb{C}^*$-actions as in the present work. Note that, to obtain a Mori chamber decomposition, some extra properties are required. We refer to \cite[Theorem 4.2 and Theorem 5.2]{KKL} for the exact statements.
\end{remark}

We introduce a special type of MDR that we will be most interested in:

\begin{definition}\label{def:MDRtype}
Let $Y_-$ be a normal projective variety,  and let
$L_-,L_+\in\CDiv(Y_-)$ be two effective Cartier divisors such that:
\begin{itemize}
\item $L_-$ is ample,
\item $R(Y_-;L_+)$ is finitely generated and the natural map $\phi:Y_-\dashrightarrow Y_+:=\Proj(R(Y_-;L_+))$ is a small modification.
\end{itemize}
We say that $(L_-,L_+)$ is a {\em Mori Dream Bond} ({\em MDB}, for short) on $Y_-$ if the cone $\cC$ generated by $L_-,L_+$ is an MDR.
\end{definition}

\subsection{Bia\l ynicki-Birula decomposition, criticality}\label{ssec:BB}

In this section, we collect some preliminaries on $\C^*$-actions on polarized pairs $(X,L)$.

Given such a pair, it is known (see \cite[Proposition 2.4]{KKLV} and the subsequent Remark) that the action of $\C^*$ on $X$ admits a linearization on the line bundle $L$. 
In particular, denoting by $\cY$ the set of connected components of the fixed point locus, we may assign to every $Y\in \cY$ a weight $\mu_L(Y)\in\Mo(\C^*):=\Hom(\C^*,\C^*)$, which indicates the weight of the restriction of the action of $\C^*$ to any fiber of $L$ over a point of $Y$. As usual we will identify $\Mo(\C^*)$ with $\Z$ and think of $\mu_L(Y)$ as integers, and we will call them the \emph{critical values} of the action. Notice that a linearization of $L$
induces a linearization of $mL$, for $m\in \Z_{>0}$, such that $\mu_{mL}(Y) = m \mu_L(Y)$, for any $Y\in \cY$.

Considering the weights $\mu_L(Y)$, with $Y\in\cY$, in an increasing order, we obtain a chain of the form
$$a_0< a_1 < \ldots < a_r.$$
We define the \emph{bandwidth} $\delta$ of the $\C^*$-action as the difference $a_r-a_0$, and define the \emph{criticality} of the action as the integer $r$. Moreover, we set $$Y_i:=\bigsqcup_{Y\in \cY,\mu_{L}(Y)=a_i} Y.$$

For every $Y\in \cY$ we denote by 
\begin{equation*}\label{eq:BBcells}
	X^+(Y):=\{x\in X\mid \lim_{t\to 0} tx\in Y\}, \hspace{0.3cm}
	X^-(Y):=\{x\in X\mid \lim_{t\to \infty} tx\in Y\}
\end{equation*}
the {\it Bia{\l}ynicki-Birula cells} of the action. 

We are not assuming the smoothness of $X$, therefore we cannot claim that the natural maps $X^\pm(Y)\to Y$ are affine bundles, as it would follow by the Bia\l ynicki-Birula theorem (see for instance \cite[Theorem 2.4]{WORS1}). However, being $X$ normal and projective and using \cite[Theorem 1]{sumihiro} we may still claim that the varieties $X^\pm(Y)$ are locally closed and, in particular, that there exists a unique fixed component $Y_+\in \cY$ (resp. $Y_-\in \cY$) such that $X^+(Y_+)$ (resp. $X^-(Y_-)$) is a dense open set in $X$. We will call $Y_+$ and $Y_-$ respectively the \textit{source} and the \textit{sink} of the action, and refer to them as the \textit{extremal fixed components}; the rest of fixed point components are called {\em inner}, and their set will be denoted by $\cY^\circ$. 
Note that $Y_\pm$ can also be described as the fixed point components where $\mu_L$ acquires its extremal values $a_0,a_r$ (see for instance \cite[Remark 2.12]{WORS1}). 

Moreover we say that a $\C^*$-action is \emph{equalized} at $Y\in \cY$ if for every point $p\in (X^-(Y)\cup X^+(Y))\setminus Y$ the isotropy group of the $\C^*$-action at $p$ is trivial. More generally, a $\C^*$-action is equalized if it is equalized at every component $Y\in \cY$. 
In the case in which $X$ is smooth, this is equivalent to say that the induced action on the normal bundle of every fixed point component has weights $\pm 1$ (see \cite[Lemma 2.1]{WORS3}).

\subsection{GIT-quotients of $\C^*$-actions and admissible quotients}\label{ssec:BBS}

Throughout this section we will assume that $(X,L)$ is a polarized pair together with an action of $\C^*$ of criticality $r$, and we will use for it the notation introduced in Section \ref{ssec:BB}. We will briefly recall the construction of geometric quotients of an action of $\C^*$ on a polarized pair $(X,L)$ (we refer to \cite[\S 2.2]{WORS3} and references therein for a detailed explanation of this construction).

The point of view of \cite[\S 2.2]{WORS3} was based on \cite{BBS}, on which the authors show how geometric quotients of $X$ are described in terms of two special types of partitions of $\cY$, called \emph{sections}. Here we will be interested in a particular example of those quotients, which are defined as the projective spectra of certain subalgebras of $R(X;L)$. 

\begin{definition}\label{def:section}
	For any index $i=0,\ldots,r-1$, the partition 
	$$
	\cY=\underbrace{\{Y\in\cY\mid \mu_L(Y)\leq a_i\}}_{\cY_-}\sqcup \underbrace{\{Y\in\cY\mid \mu_L(Y)\geq a_{i+1}\}}_{\cY_+}
	$$
	is a section in the sense of \cite[Definition 1.2]{BBS}, determining an open subset:
	$$
	X^{\st}(i,i+1):=X\setminus (\bigsqcup_{Y\in \cY_-} X^+(Y)\sqcup \bigsqcup_{Y\in \cY_+} X^-(Y)),
	$$
	which is the set of stable points of a certain linearization of the $\C^*$-action on $L$. Therefore there exists a geometric quotient %$\pi_i: X^{\st}(i,i+1)\to 
	$\GX(i,i+1):=X^{\st}(i,i+1)\git \C^*$. We will call $\GX(i,i+1)$ the $i$-th geometric quotient of the polarized pair $(X,L)$. We will refer to $\GX(0,1), \GX(r-1,r)$ as the \emph{extremal geometric quotients} of $(X,L)$, while the others will be called \emph{inner geometric quotients}.
\end{definition}	
			
	The following remark explains the relation among the GIT-quotients introduced above and the $\C$-algebra $R(X;L)$.

\begin{remark}\label{remark:projquotients}
We recall that the geometric quotients can be described as the projective spectra of some finitely generated $\C$-subalgebras of $R(X;L)$. Indeed, following \cite[Proposition 2.11]{WORS3} and \cite[Amplification 1.11, p. 40]{MFK}, choose a rational $\tau \in (a_i,a_{i+1})\cap \Q$. Then
	$$\GX(i,i+1)=\Proj \bigoplus_{m\geq 0, \\ m\tau\in\Z} \HH^0(X,mL)_{m\tau},$$
where by $\HH^0(X,mL)_{m\tau}$ we denote the direct summand of $\HH^0(X,mL)$ with weight associated to the $\C^*$-action equal to $m\tau$.
\end{remark}

\begin{definition}\label{def:semigeometricqquotient}
	For every $i=0,\ldots,r$, we can define a \emph{semigeometric quotient} $\GX(i,i)$ as the projective spectrum of the finitely generated $\C$-algebra
	$$\GX(i,i):=\Proj \bigoplus_{m\geq 0} \HH^0(X,mL)_{ma_i}.$$
	A semigeometric quotient may also be defined in terms of partitions, see for instance \cite[Proposition 2.0]{WORS3}; this presentation allows us to define, for every $i$, natural surjective morphisms:
	$$
	\xymatrix{\GX(i-1,i)\ar[rd]&&\GX(i,i+1)\ar[ld]\\&\GX(i,i)&}
	$$
	The case of our interest will be the \emph{extremal semigeometric quotients}, namely $\GX(0,0)$ and $\GX(r,r)$. There are bijective morphisms $\GX(0,0)\to Y_-$, $\GX(r,r)\to Y_+$.
\end{definition}

\begin{remark}\label{remark:definitionpsi}
	Since the intersection $\bigcap_{i=0}^{r-1} X^s(i,i+1)$ is open and non-empty, we obtain a birational map among the geometric quotients	
	$$\psi: \GX(0,1)\dashrightarrow \GX(1,2) \dashrightarrow \ldots \dashrightarrow \GX(r-1,r).$$
	The map $\psi$ is called the \emph{birational map associated to the $\C^*$-action on $(X,L)$}.
\end{remark}

\begin{definition}\label{def:btype}
	Let $(X,L)$ be a polarized pair endowed with a $\C^*$-action of criticality $r$. The action is of \emph{B-type} if the compositions $\GX(0,1)\to \GX(0,0)\to Y_-$, $\GX(r-1,r)\to \GX(r,r) \to Y_+$ are isomorphisms.
\end{definition}

\begin{remark}
The above definition tells us that $1$-dimensional orbits having a limiting point $x_\pm\in Y_\pm$ are uniquely determined by $x_\pm$, and that these two correspondences are given by isomorphisms $\GX(0,1)\to Y_-$, $\GX(r-1,r) \to Y_+$. If the variety $X$ is smooth, the Bia{\l}ynicki-Birula Theorem implies that the action is of B-type if and only if the sink and the source are subvarieties of codimension $1$, and this was the original definition of B-type action, as presented in \cite[Definition~3.1]{WORS1}. Note that in this case, the birational map $\psi$ coincides with the one introduced in \cite[Lemma 3.4]{WORS1}, which associates to every point $p\in Y_-$ the limit, for $t\to 0$, of the unique orbit having $p$ as limit for $t\to\infty$.
\end{remark}

Our next goal is to study when the birational map $\psi$ is a small modification; to this end, we introduce a new class of geometric quotients, called \emph{admissible}.

\begin{definition}
	A geometric quotient $\GX(i,i+1)$ is \emph{admissible} if $X\setminus (X^{s}(i,i+1)\cup Y_{\pm})$ does not contain codimension one subvarieties. 
\end{definition}

\begin{definition}\label{def:bordism}
	A $\C^*$-action is a \emph{bordism} if it is of $B$-type and, for every inner component $Y$, the closure of the Bia\l ynicki-Birula cells $\overline{X^{\pm}(Y)}$ does not contain codimension one subvarieties.
\end{definition}

Let us notice that Definitions \ref{def:btype}, \ref{def:bordism} generalize the ones given in \cite[Definition 3.1, Definition 3.8]{WORS1} in the smooth case.

\begin{example}\label{ex:grass}
A simple example of bordism can be constructed as follows (we refer the interested reader to \cite[\S 5.4]{WORS3}, \cite{FSC} for details). Let $V_-,V_+$ be two complex vector spaces of dimension $n$, $k\in\{2,\dots,n-1\}$, and consider the $\C^*$-action on the Grassmannian $G(k,V)$ of $k$-dimensional vector subspaces of $V:=V_-\oplus V_+$ induced by the linear action on $V$ whose weights are $0$ on $V_-$, and $1$ on $V_+$. The sink and the source of this action are, respectively, the Grassmannians $G(k,V_-)$, $G(k,V_+)$, and the inner fixed point components of the action take the form $Y_i:=G(k-i,V_-)\times G(i,V_+)$, $i=1,\dots,k-1$. The $\C^*$-action extends to the blow-up $X$ of $G(k,V)$ along $G(k,V_-)\cup G(k,V_+)$ and, following \cite{FSC} the extremal fixed point components of this action (which are the first and last geometric quotients) are, respectively,
$$
\GX(0,1)=\P_{G(k,V_-)}(\cS_-^\vee\otimes V_+),\quad \GX(k-1,k)=\P_{G(k,V_+)}(V_-\otimes \cS_+^\vee),
$$ 
where $\cS_\pm\subset V_\pm\otimes \cO_{G(k,V_\pm)}$ denotes the universal (rank $k$) subbundle of $G(k,V_\pm)$. This action is an equalized bordism, inducing a small birational transformation: 
$$
\xymatrix{\GX(0,1)
%=\P_{G(k,V_-)}(\cS_-^\vee\otimes V_+)
\ar@{-->}[rr]^(.45){\psi}&&
\GX(k-1,k)
%=\P_{G(k,V_+)}(V_-\otimes \cS_+^\vee)
.}
$$ 
Following \cite{WORS3} $\psi$ is the composition of $(k-1)$ small modifications:
$$
\xymatrix{\GX(0,1)\ar@{-->}[r]^{\psi_1}&\GX(1,2)\ar@{-->}[r]^(.50){\psi_2}&\dots\ar@{-->}[r]^(.30){\psi_{k-1}}&\GX(k-1,k).}
%
%\GX(0,1)\stackrel{\psi_1}{\dashrightarrow} \GX(1,2) \stackrel{\psi_1}{\dashrightarrow} \ldots \stackrel{\psi_{k-1}}{\dashrightarrow} \GX_{k-1} \stackrel{\psi_{k}}{\dashrightarrow} \GX_+.
$$
In this particular example all the geometric quotients $\GX(i,i+1)$ are smooth varieties (cf. \cite[Lemma~2.14]{WORS3}), and the maps $\psi_i$ are Atiyah flips. Each $\psi_i$ has as indeterminacy locus a projective bundle $\P(\cN^\vee_+)$ over $Y_i$ into $\GX(i-1,i)$, and substitutes it by another projective bundle $\P(\cN^\vee_-)$ over $Y_i$ into $\GX(i,i+1)$:
$$
\xymatrix@R=6mm{\GX(i-1,i)\ar@{-->}[rr]^{\psi_i}&&\GX(i,i+1)\\
\P(\cN^\vee_+)\ar[dr]\ar@{^(->}[u]&&\P(\cN^\vee_-)\ar[dl]\ar@{_(->}[u]\\
&Y_i&}
$$
The vector bundles $\cN_\pm$ are precisely the $\C^*$-eigenspaces of the normal bundle of $Y_i$ in $G(k,V)$ of weights $\pm 1$, respectively.
\end{example}

\begin{remark}\label{rem:bordismsmallmodification}
	Given a $\C^*$-action on $(X,L)$ which is a bordism, then the natural birational map $\psi:\GX(0,1)\dashrightarrow \GX(r-1,r)$ is a small modification.
\end{remark}

\begin{remark}\label{remark:bordismiffadmissible}
	Given a $\C^*$-action on $(X,L)$, every geometric quotient is admissible if and only if for every component $Y\in \cY^\circ$ it holds $\codim \overline{X^{\pm}(Y)}\geq 2$.
\end{remark}

\begin{lemma}\label{lemma:admissibleinterval}
	If $\GX(i,i+1)$ is not admissible, then either every $\GX(k,k+1)$, for $k< i$, or every $\GX(m,m+1)$, for $m>i$, is not admissible.
\end{lemma}

\begin{proof}	
By assumption, $X\setminus (X^s(i,i+1)\cup Y_{\pm})$ contains a divisor. Thus by construction such a divisor will be contained either in the closure of a cell $X^{+}(Y_j)$, if $j\leq i $, or in the closure of $X^-(Y_j)$, if $j \geq i+1$. Let us prove the statement in the first case; the second one is analogous. By definition, for any $m>i$ the  cell $X^+(Y_j)$ will be contained in $X\setminus (X^s(m,m+1)\cup Y_{\pm})$, proving that any other quotient $\GX(m,m+1)$ will not be admissible. 
\end{proof}

\begin{corollary}\label{cor:admissiblesinksource}
	If $\GX(0,1)$ and $\GX(r-1,r)$ are admissible, then every geometric quotient $\GX(i,i+1)$, for $i=1,\ldots,r-2$ is admissible too.
\end{corollary}

%\begin{proof}
%	If by contradiction there exists an index $j\in\{1,\ldots,r-2\}$ such that $\GX(j,j+1)$ is not admissible, then thanks to the previous Lemma  either every quotient $	GX(k,k+1)$, for $k\leq j$, or every quotient $\GX(s,s+1)$, for $s\geq j$ would not be admissible, hence a contradiction.
%\end{proof}

\begin{corollary}\label{coroll:bordismiffadmissible}
	A B-type $\C^*$-action is a bordism if and only if every geometric quotient is admissible.
\end{corollary}

\begin{definition}\label{definition:extension} 
	Let $(X,L)$ be a polarized pair endowed with a $\C^*$-action of criticality $r$, and assume that the action is a bordism. For every index $i=0,\ldots,r-1$, we define an \emph{extension map} as $e_i:\Div(\GX(i,i+1))\to\Div(X)$, $D\mapsto e_i(D)=\overline{\pi^{-1}_i(D)}$,
	where $\pi_{i}:X^s(i,i+1)\to \GX(i,i+1)$ is the geometric quotient map.
\end{definition}

\begin{lemma} \label{lemma:extensionrational}
	Let $(X,L)$ be as in Definition \ref{definition:extension}. %
	%be a polarized pair endowed with a $\C^*$-action of criticality $r$, and assume that the action is a bordism. 
	For every $i=0,\ldots,r-1$ and every rational function $f\in \C(\GX(i,i+1))$, we have $e_i(\div(f))=\div(f\circ \pi_i)$.
\end{lemma}

\begin{proof}
	We have $\div(f\circ \pi_i)=\overline{\pi_i^{-1}(\div(f))}+E$, where $E$ is a Weil divisor whose support is contained in $X\setminus X^s(i,i+1)$. Since $\GX(i,i+1)$ is admissible by Corollary \ref{coroll:bordismiffadmissible}, $E=0$.
\end{proof}

\begin{lemma}\label{lemma:extensionlinearequivalence}
	Let $(X,L)$ be as in Definition \ref{definition:extension}. Then for any $D,D'\in \Div(\GX(i,i+1))$ such that $D\sim D'$, it holds that $e_i(D)\sim e_i(D')$.
\end{lemma}

\begin{proof}
	Suppose that $D'=D+\div(f)$. Then, using Lemma \ref{lemma:extensionrational}, we obtain:
	\begin{equation*}
		\begin{split}
			e_i(D')&=\overline{\pi_i^{-1}(D')}=\overline{\pi_i^{-1}(D+\div(f))}=\\
			&=\overline{\pi_i^{-1}(D)}+\overline{\pi_i^{-1}(\div f)}=e_i(D)+e_i(\div(f)).
		\end{split}
	\end{equation*} 
\end{proof}

\begin{lemma}\label{lemma:extensioninvariantdivisors}
	Let $\C^*$ act on the polarized pair $(X,L)$. Then every Cartier divisor is linearly equivalent to a $\C^*$-invariant divisor. Moreover, for every $i= 0,...,r-1$, the action of $\C^*$ on $X$ is a bordism if and only if the only $\C^*$-invariant divisors are $Y_{\pm}$ and the divisors of the form $e_i(E)$, for $E\in \Div(\GX(i,i+1))$.
\end{lemma}

\begin{proof}
	Since every Cartier divisor is difference of two very ample divisors, we prove that every very ample divisor is linearly equivalent to a $\C^*$-invariant one. Let us consider the induced $\C^*$-action on the linear system $|A_1|$, with $A_1$ very ample; the action will have at least a fixed point, which is associated to a $\C^*$-invariant divisor, hence we conclude.
	
	We now show the second part of the statement, noting that the only if part is obvious. 
Let us then assume that the $\C^*$-action on $X$ is a bordism. Note that the divisors of the form  $e_i(E)$, $E\in \Div(\GX(i,i+1))$ are clearly $\C^*$-invariant. Now let $D$ be an irreducible $\C^*$-invariant divisor. If $D$ is pointwise fixed by the action, then it is either the sink or the source, by definition of bordism. On the other hand, if $D$ is $\C^*$-invariant but not pointwise fixed, then it contains an $(n-2)$-dimensional family of $1$-dimensional orbits (whose union is dense in $D$). Let $\C^*p$ be the general element of this family, and let $Y_1$, $Y_2$ be the fixed point components of the action containing the sink and the source of $\C^*p$, respectively. It follows that $D\subset X^-(Y_1)\cap X^+(Y_2)$, and from the definition of bordism we conclude that $Y_1=Y_-$, $Y_2=Y_+$. 
In particular, the general orbit in $D$ is contained in $X^{\st}(i,i+1)$, for every $i$, and we may conclude that $D=e_i((D\cap X^{\st}(i,i+1))\git \C^*)$.
%It then easily follows that $D$ can be written as divisor of the form $e_i(E)$, for $E\in \Div(\GX(i,i+1))$. 
\end{proof}

We conclude this section by recalling the notion of \emph{geometric realization of a birational map}, which was already introduced in \cite[Definition 2.10]{WORS4}:

\begin{definition}\label{def:geomrealization}
	Given a birational map $\phi: Y_-\dashrightarrow Y_+$ between normal projective varieties, a \emph{geometric realization of $\phi$}  is a normal projective variety $X$, endowed with a $\C^*$-action of B-type such that the sink and the source are precisely $Y_-, Y_+$ and the natural birational map among them, defined in Remark \ref{remark:definitionpsi}, coincides with $\phi$. 
\end{definition}

As we will later see, the properties we are requiring to a geometric realization guarantee that the birational transformation $\phi$ can be accurately described in terms of features of the $\C^*$-action. In particular, up to the choice of a polarization of $X$, the corresponding GIT quotients will give us a factorization of $\phi$. 

In Section \ref{sec:MDR2} we will address the question of constructing geometric realizations of small modifications; more concretely our approach will show the existence of a strong relation among small modifications defined by MDBs and bordisms.

%% file: constructionbordism.tex
%!TEX root = BRUS.tex

\section{Pruning a $\C^*$-action}\label{sec:constructionbordism}

In this section we show how to construct, starting from a $\C^*$-action on a polarized pair $(Z,E)$ satisfying certain assumptions, $\C^*$-equivariant birational maps $Z\dashrightarrow X$ to normal projective varieties $X$ supporting $\C^*$-actions that are bordisms: this algebro-geometric procedure will be called the \emph{pruning} of the $\C^*$-action on $Z$ (see Definition \ref{def:pruning}).

As we have seen in Definition \ref{def:bordism}, a bordism requires an isomorphism between the extremal fixed point components and the extremal geometric quotients, and the non existence of $\C^*$-invariant divisors other than the extremal fixed point components and the extensions (see Lemma \ref{lemma:extensioninvariantdivisors}). At least in the smooth case, a variety with a $\C^*$-action can be birationally modified --by means of a $\C^*$-equivariant blow-up (see \cite[Definition 3.14]{BarbanRomano}) along the extremal fixed point components-- in order to fulfill the first property. On the other hand, many examples show that the existence of inner $\C^*$-invariant divisors is common. For instance, this is always the case when the variety $Z$ has an isolated extremal fixed point (cf. \cite[Lemma~2.6]{WORS1}, see also \cite[Theorem 4.5]{RW} for particular examples of varieties with this feature). Then one can conclude that, even in the smooth case, performing a $\C^*$-equivariant blow-up does not turn a $\C^*$-action on a variety into a bordism. Pruning is then a generalization of that procedure, that allows us to birationally and $\C^*$-equivariantly modify a $\C^*$-action on a variety with  into a bordism. Intuitively the pruning procedure consists of cutting the segment of weights of the $\C^*$-action on $Z$ into a smaller interval and thus removing some fixed components of certain weights: 
\vspace{0.2cm}
\begin{center}
	\includegraphics[scale=0.15]{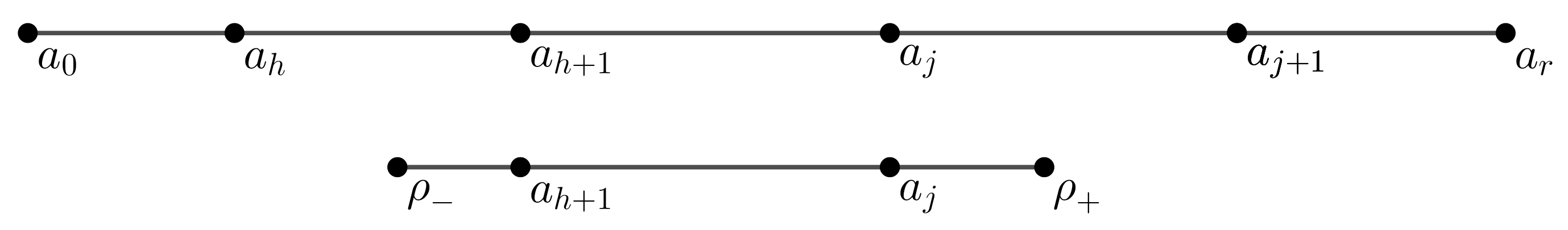}
\end{center}
The first segment represents the weights associated to the fixed components of the $\C^*$-action on $(Z,E)$. In order to obtain a bordism, we remove the components with weights smaller than or equal to (resp. greater than or equal to) $a_h$ (resp. $a_{j+1}$), so that every geometric quotient kept is admissible, and the sink and the source of the $\C^*$-action on the pruning $X$ of $Z$ are respectively the geometric quotients $Y_-:=\GZ(h,h+1),Y_+:=\GZ(j,j+1)$.

\begin{setup}\label{setup:bordism}
	Throughout this section $(Z,E)$ will denote a polarized pair, %with $Z$ projective and normal 
	admitting a nontrivial $\C^*$-action with critical values $a_0,\dots,a_r$ and bandwidth $\delta$. Moreover let us consider $\rho_-,\rho_+\in[a_0,a_r]\cap \Q$, 
	with $\rho_-< \rho_+$.
\end{setup}

\begin{lemma}\label{lemma:fingenS2}
	In the situation of Set-up \ref{setup:bordism}, assume furthermore that $\rho_{\pm}$ are integers, that $E$ is very ample, and that the embedding $Z\subset \P(\HH^0(Z,E))$ is projectively normal. 
	Then the $\C$-algebra 
	$$S:=\bigoplus_{m\geq 0}\bigoplus_{k=m\rho_-}^{m\rho_+}\HH^0(Z,mE)_k.$$ 
	is finitely generated. 
\end{lemma}		

\begin{proof}
	We will show that the $\C$-algebra
	$$\widetilde{S}:=\bigoplus_{m\geq 0}\bigoplus_{k=m\rho_-}^{m\rho_+}S^m\HH^0(Z,E)_k,$$
	is finitely generated, and that the natural homomorphism $\widetilde{i^*}:\widetilde{S} \to S$ is surjective.
	
	We first prove that the algebra $\widetilde{S}$ is finitely generated. To this end, since $E$ is ample using \cite[Lemma 2.4]{BWW} we may suppose that $\HH^0(Z,E)$ is generated by $s_1,\ldots,s_n$, with $s_i\in \HH^0(Z,E)_{w_i}$ for every $i$, where $w_i\in [a_0,a_r]\cap \Z$ are the weights of the induced $\C^*$-action on $\HH^0(Z,E)$. The monomials $\prod_is_i^{m_i}$ in $S^m(\HH^0(Z,E))$ belonging to $\widetilde{S}$ are those which satisfy the following system of inequalities:
	\begin{equation*}
		\begin{cases} \sum_{i=1}^n(w_i-\rho_-)m_i \geq 0, \\ \sum_{i=1}^n (\rho_+-w_i)m_i \geq 0, \\ m_i \geq 0.
		\end{cases} 
	\end{equation*}	
	This is a rational polyhedral cone in $\R^n$, therefore by Gordan's Lemma (see, for instance,  \cite[Proposition 1.2.17]{CLS}) its intersection with the lattice of monomials is finitely generated. 
	
	Finally, in order to prove that $\widetilde{i^*}$ is surjective we simply note that the natural map $i^*:\Sym(\HH^0(Z,E))\to \bigoplus_{m\geq 0}\HH^0(Z,mE)$ is surjective --thanks to the projective normality of $Z\subset \P(\HH^0(Z,E))$-- and $\C^*$-equivariant. 
\end{proof}

We remark that the Lemma above holds in a greater generality, but we have presented it in this way for the sake of simplicity.

\begin{definition}\label{def:pruning}
	In the situation of Set-up \ref{setup:bordism},	let $d\in \Z_{>0}$ be the minimum positive integer such that $\rho_\pm d\in \Z$. We define the {\em pruning of $(Z,E)$  with respect to $\rho_-,\rho_+$} as:
	$$
	X:=\Proj S^{(nd)},\quad n\gg 0,%\quad 
	$$
	where $S^{(nd)}$ is the graded $\C$-algebra $S^{(nd)}=\bigoplus_{m\geq 0}S^{(nd)}_m$ whose graded pieces are defined by 
	$$%S^{(nd)}=\bigoplus_{m\geq 0}S^{(nd)}_m, \quad
	S^{(nd)}_m:=\bigoplus_{k=mnd\rho_-}^{mnd\rho_+}\HH^0(Z,mndE)_k, \quad m\geq 0.
	$$
	\end{definition}
\begin{remark}\label{rem:pruning} Note that $S^{(nd)}$ is finitely generated for $n\gg 0$ by Lemma \ref{lemma:fingenS2}, and that $\Proj S^{(nd)}=\Proj S^{(n'd)}$ for $n,n'\gg 0$. Therefore $X$ is well-defined and depends only on the pair $(Z,E)$ and on the rational numbers $\rho_-,\rho_+$. Furthermore, the pruning of $(Z,E)$ with respect to $\rho_-,\rho_+$ is equal to the pruning of $(Z,nE)$ with respect to $n\rho_-,n\rho_+$, for any $n>0$. 
\end{remark}

\begin{remark} %\lstodo{Do we want to keep this remark? }
In the case in which $\rho_\pm$ belong to the same open interval $(a_i,a_{i+1})$, then the resulting variety will be a $\P^1$-fibration over the geometric quotient $\GZ(i,i+1)$% (a $\P^1$-bundle in the case in which the action on $Z$ is equalized)
, whose fibers are the closures of the $1$-dimensional orbits of the induced $\C^*$-action. The sink and the source of the action are two sections of the fibration. % the projectivization of a decomposable rank two vector bundle on $\GZ(i,i+1)$, the action will be the fiberwise natural one, whose sink and source are two sections of the $\P^1$-bundle. 
Let us notice also that if $Z$ is smooth and $\rho_-\in (a_0,a_1)$, $\rho_+\in (a_{r-1},a_r)$, then the pruning of $Z$ with respect to $\rho_-,\rho_+$ is precisely the $\C^*$-equivariant blow-up of $Z$ along the sink and the source (see \cite[Definition 3.14]{BarbanRomano}).
\end{remark}

\begin{theorem}\label{theorem:Xbordism}
	In the situation of Set-up \ref{setup:bordism}, take $\rho_-\in (a_h,a_{h+1})\cap \Q$, $\rho_+\in (a_{j},a_{j+1})\cap \Q$ for some $h, j\in\{0,\dots,r-1\}$. Then the pruning $X$ of $(Z,E)$ with respect to $\rho_-,\rho_+$ is a normal %$\Q$-factorial 
	projective variety, endowed with a B-type $\C^*$-action whose sink and source are, respectively, $\GZ(h,h+1), \GZ(j,j+1)$. Moreover $X$ is $\C^*$-equivariantly birational to $Z$ and, if $\codim \overline{Z^{\pm}(Y)}\geq 2$ for every fixed point component $Y$ of $Z$ with $h<\mu_E(Y)<j$, then the $\C^*$-action on the pruning $X$ is a bordism. 
\end{theorem}

The proof of Theorem \ref{theorem:Xbordism} will be divided in several steps. Without loss of generality, using Remark \ref{rem:pruning}, we may assume --by exchanging $E$ with a suitable multiple-- that $d=n=1$ and $\rho_{\pm}\in \Z$. Note that, by construction, the action of $\C^*$ on $R(Z;E)$ restricts to an action on $$S=\bigoplus_{m\geq 0}\bigoplus_{k=m\rho_-}^{m\rho_+}\HH^0(Z,mE)_k\subset R(Z;E),$$ providing a $\C^*$-action on $X=\Proj(S)$ such that the natural map $Z\dashrightarrow X$ is $\C^*$-equivariant.

%\begin{step2}\label{step:potaturaextremalquotients}
%	The pruning $X$ contains $\GZ(h,h+1)$ and $\GZ(j,j+1)$. 
%\end{step2}
%
%\begin{proof}
%	We prove the result for $\GZ(h,h+1)$; the case of $\GZ(j,j+1)$ is analogous. It suffices to notice that we have a surjective morphism $$\bigoplus_{m\geq 0}\bigoplus_{k=m\rho_-}^{m\rho_+} \HH^0(Z,mE)_k \to \bigoplus_{m\geq 0} \HH^0(Z,mE)_{m\rho_-},$$ with kernel the ideal $\bigoplus_{m\geq 0}\bigoplus_{m\rho_- <k \leq m\rho_+} \HH^0(Z,mE)_k$, therefore taking the projective spectrum we obtain an inclusion of varieties:  $$\GZ(h,h+1)=\Proj \bigoplus_{m\geq 0}\HH^0(Z,mE)_{m\rho_-}\hooklongrightarrow X.$$
%\end{proof}	

Along %the rest of 
the proof, we will use the following notation. For every $m>0$, we will consider the decomposition $S_m=S_m^-\oplus S_m^0 \oplus S_m^+$, where
	$$
	%S_m=S_m^-\oplus S_m^0 \oplus S_m^+,\quad 
	S_m^\pm:= \HH^0(Z,mE)_{m\rho_{\pm}},\quad 
	S_m^0:=\bigoplus_{m\rho_-<k<m\rho_+}\HH^0(Z,mE)_{k}.
	$$
For every homogeneous element $f\in S_m$, $m>0$, we will denote $D^+(f,X):=	\Spec\left(S_{(f)}\right)\subset X.$ Then we define the following open subsets of $X$:
	\begin{equation*}\label{proof:normality}
	U_{\pm}:= \bigcup_{m> 0}\bigcup_{f\in S_m^\pm} D^+(f,X),\quad 
	U_0:=\bigcup_{m> 0}\bigcup_{f\in S_m^0} D^+(f,X),
%		\begin{split}
%			U & =\bigcup_{m> 0}\bigcup_{f\in S_m^0} \Spec\left(S_{(f)}\right) \\
%			U_{\pm}& = \bigcup_{m> 0}\bigcup_{f\in S_m^\pm} \Spec\left(S_{(f)}\right).
%		\end{split}
	\end{equation*}
and note that, by construction, $X=U_-\cup U_0 \cup U_+$, and that $U_0,U_\pm$ are $\C^*$-invariant.

\begin{step2}\label{step:potaturanormal}
	The variety $X=\Proj(S)$ is normal.
\end{step2}

\begin{proof}
%	Let us set, for every $m>0$, the decomposition $S_m=S_m^-\oplus S_m^0 \oplus S_m^+$, where
%	$$
%	%S_m=S_m^-\oplus S_m^0 \oplus S_m^+,\quad 
%	S_m^\pm:= \HH^0(Z,mE)_{m\rho_{\pm}},\quad 
%	S_m^0:=\bigoplus_{m\rho_-<k<m\rho_+}\HH^0(Z,mE)_{k}.
%	$$
%	We then define the following open subsets of $X$:
%	\begin{equation*}\label{proof:normality}
%	U_{\pm}:= \bigcup_{m> 0}\bigcup_{f\in S_m^\pm} \Spec\left(S_{(f)}\right),\quad 
%	U_0:=\bigcup_{m> 0}\bigcup_{f\in S_m^0} \Spec\left(S_{(f)}\right).
%%		\begin{split}
%%			U & =\bigcup_{m> 0}\bigcup_{f\in S_m^0} \Spec\left(S_{(f)}\right) \\
%%			U_{\pm}& = \bigcup_{m> 0}\bigcup_{f\in S_m^\pm} \Spec\left(S_{(f)}\right).
%%		\end{split}
%	\end{equation*}
%%	\begin{equation*}\label{proof:normality}
%%		\begin{split}
%%			U & =\bigcup_{f\in \bigoplus_{m\geq 0}\bigoplus_{m\rho_-<k<m\rho_+}\HH^0(Z,mE)_{k}} \Spec\left(\bigoplus_{m\geq 0}\bigoplus_{k=m\rho_-}^{m\rho_+} \HH^0(Z,mE)_{k}\right)_{(f)} \\
%%			U_{\pm}& = \bigcup_{f\in \bigoplus_{m\geq 0}\HH^0(Z,mE)_{m\rho_{\pm}}} \Spec\left(\bigoplus_{m\geq 0}\bigoplus_{k=m\rho_-}^{m\rho_+} \HH^0(Z,mE)_{k}\right)_{(f)}.
%%		\end{split}
%%	\end{equation*}
%	By construction, $X=U_-\cup U_0 \cup U_+$; w
We will show that the affine open subsets $D^+(f,X)\subset X$ are normal, for every $f\in S_m^-\cup S_m^0 \cup S_m^+$. 

Let us start with the case in which $f\in S_m^0$. 
We claim that \begin{equation}S_{(f)}=R(Z;E)_{(f)}\label{eq:local}\end{equation}
	%\left(\bigoplus_{m\geq 0}\bigoplus_{k=m\rho_-}^{m\rho_+}\HH^0(Z,mE)_k\right)_{(f)} = \left(\bigoplus_{m\geq 0} \bigoplus_{k=0}^{m\delta} \HH^0(Z,mE)_k\right)_{(f)}.$$
	The ``$\subset$'' inclusion is obvious by construction, let us prove the converse. Given an element $\frac{g}{f^a}\in R(Z;E)_{(f)}$, with $g,f^a\in S_{ma}$, 
	%Thanks to the induced $\C^*$-action on the global sections, 
	we can decompose $g=\sum_{k=0}^{ma\delta} g_k$, with $g_k\in \HH^0(Z,mE)_k$. There exists a suitable $l\geq 0$  -- it is enough to take $l\geq ma\rho_-,ma(\delta-\rho_+)$-- for which 
	$$f^lg\in S_{m(a+l)}^0,
	%\bigoplus_{m> 0}S^0_m,
	\mbox{ therefore } \dfrac{f^lg}{f^{l+a}}\in \left(	\bigoplus_{m> 0}S^0_m\right)_{(f)}, $$ 	
	%$$f^lg\in \bigoplus_{m\geq 0}\bigoplus_{k=m\rho_-}^{m\rho_+} \HH^0(Z,mE)_k,\mbox{ therefore } \dfrac{f^lg}{f^{l+a}}\in \left(\bigoplus_{m\geq 0}\bigoplus_{k=m\rho_-}^{m\rho_+}\HH^0(Z,mE)_k\right)_{(f)}, $$ 
	thus we obtain the other inclusion. This tells us that $D^+(f,X)$ is isomorphic to an open subset $D^+(f,Z):=\Spec\big(R(Z;E)_{(f)}\big)$ of $Z$, hence normal. %Therefore the two open affine subsets of $X$ and $Z$ obtained as $\Spec$ of these algebras are isomorphic, and since every open subset of $Z$ is normal, we conclude that $D^+(f)$ is normal. 
	
Next we prove that $D^+(f,X)$ is normal for every $f\in S_m^-$ (the proof for $f\in S_m^+$ is analogous). Note that in this case the inclusion $S_{(f)}\subset R(Z;E)_{(f)}$ is not an equality in general, but an argument analogous to the one above tells us that:
$$
R(Z,E)_{(f)}=\left(S'\right)_{(f)} ,\qquad \mbox{where } S':=\bigoplus_{m\geq 0}\bigoplus_{k=0}^{m\rho_+}\HH^0(Z,mE)_k.
$$

Let us now consider a polynomial ring in one variable $\C[y]$, and consider the $\C^*$-action on it given by $t\cdot(\sum_bc_by^b)=\sum_bc_bt^{-b}y^b$. We then consider the induced $\C^*$-action on the $\C$-algebra $\ol{S}:=S'_{(f)}\otimes_{\C}\C[y]= S'_{(f)}[y]$. Note that the variety $\Spec(\ol{S})=D^+(f,Z)\times \C$ is normal, and so it is its categorical quotient by the induced $\C^*$-action (cf. \cite[Chap. 0, {\S}2, (2)]{MFK}), which is $\Spec\big(\ol{S}^{\C^*}\big)$. 

We may then conclude by noting that we have an isomorphism $\varphi:\ol{S}^{\C^*}\to S_{(f)}$. In fact, every element of $S'^{\C^*}$ can be written as a finite sum of the form:
$$
\sum_{b=0}^{ma(\rho_+-\rho_-)}\dfrac{g_{b}}{f_a}y^b, \mbox{ where }g_b\in \HH^0(Z,maE)_{ma\rho_-+b}. %\quad ma\rho_-+b\leq ma\rho_+.
$$
The required isomorphism is then given by $\varphi\left(\sum_{b\geq 0}\frac{g_{b}}{f_a}y^b\right)=\sum_{b\geq 0}\frac{g_{b}}{f_a}$.
% Notice that we have a natural morphism
%	$$U_-\to \bigcup_{f\in \bigoplus_{m\geq 0} \HH^0(Z,mE)_{m\rho_-}}\Spec\left(\bigoplus_{m\geq 0}\HH^0(Z,mE)_{m\rho_-}\right)_{(f)}$$
%	given by the corresponding inclusion of algebras; moreover the latter is an open cover of $\GZ(h,h+1)$. Since $\GZ(h,h+1)$ is a geometric quotient of $Z$, hence normal by \cite[Chap. 0, \S 2, (2)]{MFK}, $U_-$ is a $\C$-bundle on a normal variety, hence it is normal as well.  
\end{proof}

\begin{step2}\label{step:potaturabirational}
	The natural $\C^*$-equivariant map $Z\dashrightarrow X$ is birational.
\end{step2}

\begin{proof}
It suffices to notice that, as in the proof of Step~\ref{step:potaturanormal}, the inclusion of graded $\C$-algebras $S\subset R(Z;E)$ induces isomorphisms $S_{(f)}\simeq R(Z;E)_{(f)}$ for every $f\in S_m^0$, $m>0$. In particular the induced rational map $Z\dashrightarrow X$ sends the affine open set $D^+(f,Z)\subset Z$ isomorphically onto $D^+(f,X)\subset X$. Note that this in particular tells us that the open set $U_0\subset X$ introduced above is the isomorphic image of the open subset $\bigcup_{m>0}\bigcup_{f\in S_m^0}D^+(f,Z)\subset Z$. 
\end{proof}

Since the algebra $S$ is finitely generated by Lemma \ref{lemma:fingenS2}, there exists a positive integer $d'$ such that $S^{(d')}=\bigoplus_{m\geq 0} S_{d'm}$ is generated in degree $1$. Therefore $X\subset \P^N:=\P(S_{d'})$, and let us denote by $L=\cO_{\P^N}(1)_{|X}$. Since $X$ is normal, $L$ is $\C^*$-linearizable. For the rest of the section we will consider the $\C^*$-action on the polarized pair $(X,L)$.

\begin{step2}\label{step:potaturafixedpointlocus}
	The sink and the source of the $\C^*$-action on the pruning $X$ of $Z$ are isomorphic to  $\GZ(h,h+1),\GZ(j,j+1)$, respectively. The inner fixed components of $X$ are isomorphic to the fixed point components of $Z$ of weight equal to $a_{h+1},\ldots,a_j$. Furthermore, the criticality of the induced $\C^*$-action on $(X,L)$ is equal to $j-h+1$. 
\end{step2}

\begin{proof}
Note first that we have a surjective homomorphism of $\C$-algebras:
$$S=\bigoplus_{m\geq 0}\bigoplus_{k=m\rho_-}^{m\rho_+} \HH^0(Z,mE)_k \to \bigoplus_{m\geq 0} \HH^0(Z,mE)_{m\rho_-},$$ %whose kernel is the ideal $\bigoplus_{m\geq 0}\bigoplus_{m\rho_- <k \leq m\rho_+} \HH^0(Z,mE)_k$. Thus, taking the projective spectrum we obtain 
which an inclusion of varieties:  $$\GZ(h,h+1)=\Proj \bigoplus_{m\geq 0}\HH^0(Z,mE)_{m\rho_-}\hooklongrightarrow X.$$
By construction, $\GZ(h,h+1)$ is fixed by the $\C^*$-action. Moreover, using \cite[Remark 2.12]{WORS1}, the induced $\C^*$-action on the projective space $\P^N=\P(S_{d'})\supset X$ defined above  has sink $\P(\HH^0(Z,d'E)_{d'\rho_-})$. Then we may conclude that $\GZ(h,h+1)\subset X$ is the sink of $X$ by noting that $\P(\HH^0(Z,d'E)_{d'\rho_-})\cap X=\GZ(h,h+1)$. In a similar way, one may prove that the source of $X$ is isomorphic to $\Proj \bigoplus_{m\geq 0}\HH^0(Z,mE)_{m\rho_+}\simeq \GZ(j,j+1)$.

In order to compute the inner fixed point components of $X$, we note first that the complement of the extremal fixed point components of $X$ is the open set $U_0=\bigcup_{m> 0}\bigcup_{f\in S_m^0} D^+(f,X)$, which is $\C^*$-equivariantly isomorphic to an open set 
$\bigcup_{m> 0}\bigcup_{f\in S_m^0} D^+(f,Z)\subset Z$ (see Step~\ref{step:potaturabirational}), whose fixed point components are the fixed point components of $Z$ of $L$-weight $\mu_L(Y)\in\{a_{h+1},\dots, a_j\}$.  

Finally, considering the embedding $X\subset \P^N$, the inner fixed point components of $X$ are the irreducible components in the intersections $X\cap \P(\HH^0(Z,d'E)_{d'a_i})$, $i\in\{h+1,\dots,j\}$. In particular, the criticality of the $\C^*$-action on the polarized pair $(X,L)$ is  $j-h+1$. 
\end{proof}

\begin{step2}\label{step:potaturabtype}
	The $\C^*$-action on $(X,L)$ is of B-type.
\end{step2}

\begin{proof}
	Using Remark \ref{remark:projquotients} and the arguments above one may show that, for every $i=h+1,\ldots,j-1$, it holds $\GX(i-h,i-h+1)\simeq \GZ(i,i+1)$. Therefore since $\GZ(0,1)$ and $\GZ(r-1,r)$ are the sink and the source of the $\C^*$-action in $X$ by Step~\ref{step:potaturafixedpointlocus}, we conclude.
\end{proof}

\begin{step2}\label{step:potaturabordism}
	If $\codim \overline{Z^{\pm}(Y)}\geq 2$ for every fixed point component $Y$ of $Z$ with $h<\mu_E(Y)<j$, then the $\C^*$-action on the pruning $X$ is a bordism.  
\end{step2}

\begin{proof}
	For every fixed point connected component $Y$ of $Z$ of weights $h< \mu_E(Y)< j$, using Steps \ref{step:potaturabirational} and \ref{step:potaturafixedpointlocus} we conclude $\codim \overline{X^{\pm}(Y)} \geq 2$, i.e. that the $\C^*$-action on $X$ is a bordism.
\end{proof}

Using Steps \ref{step:potaturanormal}, \ref{step:potaturabirational}, \ref{step:potaturafixedpointlocus}, \ref{step:potaturabtype} 
and \ref{step:potaturabordism} we conclude the proof of Theorem \ref{theorem:Xbordism}. 
Note that the second part of this statement can also be rephrased as follows:

\begin{corollary}
	In the situation of Set-up \ref{setup:bordism}, if $\codim \overline{Z^{\pm}(Y)}\geq 2$ for every fixed point component $Y$ with $h<\mu_E(Y)<j$, then the pruning $X$ is a geometric realization of the small modification $\psi: \GZ(h,h+1)\dashrightarrow \GZ(j,j+1)$.
\end{corollary}

%% file: MDR2.tex
%!TEX root = BRUS.tex

\section{From MDBs to bordisms}\label{sec:MDR2}

The purpose of this section is to construct a geometric realization %(see Definition \ref{def:geomrealization}) 
of the small modification $\phi: Y_-\dashrightarrow Y_+$ associated to an MDB $(L_-,L_+)$ (recall Definition \ref{def:MDRtype}). More precisely, we show the following:

\begin{theorem}\label{thm:MDR2}
	Let $Y_-$ be a normal projective variety, and let $(L_-,L_+)$ be an MDB on $Y_-$. We denote by $\phi:Y_-\dashrightarrow Y_+$ the associated small modification. Then there exists a geometric realization of $\phi$, whose induced $\C^*$-action is a bordism.
\end{theorem}

\subsection{Construction of the geometric realization}\label{ssec:construction}

For the sake of notation, in this section we will denote by 
$$A:=R(Y_-;L_-,L_+)=\bigoplus_{m_\pm\geq 0} \HH^0(Y_-,m_-L_-+m_+L_+)$$
the finitely generated $\C$-algebra associated to the MDB $(L_-,L_+)$.
The group $$H:=\Hom(\Z(L_-,L_+),\C^*)$$ is a $2$-dimensional complex torus, whose character lattice $\Mo(H)$ is $\Z(L_-,L_+)$, and that acts naturally on the finitely generated $\C$-algebra $A$. Let $\alpha\in \Mo(H)^\vee$ be a $1$-parametric subgroup, such that $\alpha_{\pm}:=\alpha(L_{\pm})>0$; we assume, for simplicity, that $\alpha_{\pm}$ are coprime, and denote by $H'\subset H$ the associated $1$-dimensional subtorus. 
We will later show (see Remark \ref{remark:birational1ps}) that different choices of $\alpha\in \Mo(H)^\vee$ yield birationally equivalent geometric realizations. 

\begin{remark}
	The $1$-dimensional torus $H'$ acts on $\Spec A$, inducing an $\N$-grading on $A$. The algebra $A$ endowed with this grading will be denoted by $A^\alpha$: 
	$$A^\alpha:=\bigoplus_{m\geq 0} A^\alpha_m, \hspace{0.3cm} \text{ where } \hspace{0.3cm} A^\alpha_m:=\bigoplus_{\substack{m_\pm\in\Z_{\geq 0}\\\alpha(m_-L_-+m_+L_+)=m}} \HH^0(Y_-,m_-L_-+m_+L_+).$$
\end{remark}

\begin{definition}\label{def:Xalpha}
	The $\N$-graded algebra $A^{\alpha}$ is finitely generated by assumption, so we may define $X^\alpha:=\Proj A^\alpha$.
\end{definition}

\begin{remark}
	Given $\alpha\in \Mo(H)^\vee$ and $H'\subset H$ as above, we may consider the $1$-dimensional torus
	$$
	H'':=H/H',
	$$
	which acts on $X^\alpha$. Indeed the $H$-action on $A^\alpha$ induces an $H$-action on $X^\alpha$, whose kernel is precisely $H'$.  
\end{remark}

\begin{remark}\label{rem:LocalCox}
	We notice %\lbtodo{to discuss} 
	that $A$ can be thought of as a local version of the Cox ring construction for MDSs, in the following sense: under the usual MDS-like point of view, one considers the affine variety $\Spec(A)$, from which $Y_\pm$ (as well as some other birationally equivalent varieties) can be obtained as geometric quotients by the action of the torus $H$. But all these geometric quotients factor via the projection map $\Spec(A)\dashrightarrow X^\alpha=\Proj(A^\alpha)$, and so the geometric quotients of $\Spec(A)$ via the action of $H$ are geometric quotients of $X^\alpha$ by the action of $H''$; the converse of this is obviously true. We summarize this construction by means of the following diagram:
	
	\adjustbox{scale=1,center}{%
		\begin{tikzcd}
			&  & \Spec A \arrow[d, "/H'", dashed] \arrow[lldd, "/H"', dashed] \arrow[rrdd, "/H", dashed]                    &  &     \\
			&  & X^\alpha=\Proj A^\alpha \arrow[d, "/H''", dashed] \arrow[lld, "/H''", dashed] \arrow[rrd, "/H''"', dashed] &  &     \\
			Y_- \arrow[rr, dashed] &  & {\GX(i,i+1)} \arrow[rr, dashed]                                                                            &  & Y_+
		\end{tikzcd}
	}   
\end{remark}

In the following statements we will use the following $\P^1$-bundle over $Y_-$:
$$P^\alpha:=\P_{Y_-}(\alpha_+L_-\oplus\alpha_-L_+).$$

\begin{lemma}\label{lemma:PbirationaXalpha}
	The projective variety $X^{\alpha}$ is normal, and there exists a birational map $\Phi:P^\alpha\dashrightarrow X^\alpha$. 
\end{lemma}

\begin{proof} 
Let us consider the ring of sections of the tautological line bundle $\cO_{P^{\alpha}}(1)$. We have:
\begin{equation*}
\begin{split}
R(P^{\alpha};\cO_{P^{\alpha}}(1))& = \bigoplus_{m\geq 0} \bigoplus_{m_-+m_+=m} \HH^0(Y_-,m_-\alpha_+L_-+m_+\alpha_-L_+)\\
			&=\bigoplus_{m\geq 0} \bigoplus_{\alpha(m_-\alpha_+L_-+m_+\alpha_-L_+)=m\alpha_-\alpha_+} \HH^0(Y_-,m_-\alpha_+L_-+m_+\alpha_-L_+)\\
			&=\bigoplus_{m\geq 0} \bigoplus_{\alpha(m_-L_-+m_+L_+)=m\alpha_-\alpha_+} \HH^0(Y_-,m_-L_-+m_+L_+),
\end{split}
\end{equation*}
where the last equality follows from the fact that $\alpha_-,\alpha_+$ are coprime. But this $\C$-algebra is precisely the $(\alpha_-\alpha_+)$-Veronese of $A^\alpha$, which is finitely generated.
It follows that $R(P^{\alpha};\cO_{P^{\alpha}}(1))$ is finitely generated, and that 
$$
\Proj\left(R(P^{\alpha};\cO_{P^{\alpha}}(1))\right)\simeq X^\alpha.
$$
In particular (see for instance \cite[Lemma~7.10]{De}) $X^\alpha$ is normal. Furthermore, the line bundle $\cO_{P^{\alpha}}(1)$ is big (see \cite[Lemma~2.3.2.~(iv)]{L2}), hence the associated rational map $\Phi:=\varphi_{\cO_{P^{\alpha}}(1)}:P^{\alpha}\dashrightarrow X^\alpha$ is birational.
\end{proof}

\begin{remark} \label{rem:equivariant}
	Notice that the $\P^1$-bundle $P^\alpha$ admits a natural equalized $H''$-action with fixed point locus $s_-(Y_-)\sqcup s_+(Y_-)$, where $s_-,s_+$ are the sections of $P^\alpha$ over $Y_-$ corresponding respectively to the projections of $\alpha_+ L_- \oplus \alpha_- L_+\to \alpha_+ L_-$, $\alpha_+ L_- \oplus \alpha_- L_+\to \alpha_- L_+$.
	Moreover the birational map $\Phi: P^\alpha\dashrightarrow X^\alpha$ introduced in Lemma \ref{lemma:PbirationaXalpha} is $H''$-equivariant. 
\end{remark}

\begin{proposition}\label{proposition:X^alphageometricrealization}
	The action of $H'' \simeq  \C^*$ on $X^{\alpha}$ is B-type with sink and source $Y_-$ and $Y_+$,  respectively, and the induced natural birational map $\psi:Y_-\dashrightarrow Y_+$ coincides with  the small modification $\phi\colon Y_-\dashrightarrow Y_+$.
\end{proposition}

\begin{proof} 
	Consider the birational map $\Phi:P^\alpha\dashrightarrow X^\alpha$ introduced in Lemma \ref{lemma:PbirationaXalpha}. We first prove that the indeterminacy locus of $\Phi$ is contained in $s_+(Y_-)$. %$\Bs{|\cO(1)|}\subseteq s_+(Y_-)$. 
	We recall that $L_-$ is ample on $Y_-$, thus $\Phi_{\mid s_-(Y_-)}$ is well defined; and since by Remark \ref{rem:equivariant} we know that $\Phi$ is $\C^*$-equivariant, it follows that the indeterminacy locus of $\Phi$  is $H''$-invariant, hence our claim. Therefore, $\Phi_{\mid P^{\alpha}\setminus s_+(Y_-)}\colon P^{\alpha}\setminus s_+(Y_-) \to \mathcal{U}_-\subset X^{\alpha}$ is an isomorphism where $\mathcal{U}_-$ is a $\C^*$-invariant neighborhood of the sink of the action on $X^\alpha$; it follows that the sink of the $\C^*$-action on $X^\alpha$ is $s_-(Y_-)  \simeq  Y_-$ and is isomorphic to the first geometric quotient of such action. 
	
	In order to conclude that the $\C^*$-action on $X^\alpha$ is of B-type, we study an $H''$-invariant neighborhood $\mathcal{U}_+$ of the source of $X^\alpha$. To do so, we consider the $\P^1$-bundle $\widetilde{P^{\alpha}}:=\P_{Y_+}(\alpha_+ L_- \oplus \alpha_- L_+)$ on $Y_+$ and show, in a similar way as above, that we can find a neighborhood $\mathcal{U}_+$ isomorphic to the complement of a section of $\widetilde{P^{\alpha}}$.
	On the other hand, using the arguments above and the construction of the natural birational map $\psi$ among the extremal geometric quotients it follows that $\psi$ coincides with the small modification $\phi$ associated to the MDB $(L_-,L_+)$. 
\end{proof}

\begin{lemma}\label{lemma:Xalphabordism}
	The $H''$-action on $X^\alpha$ is a bordism.
\end{lemma}

\begin{proof}
	Thanks to \cite[Lemma 7.10]{De}, there exists an open subset $U$ of $X^\alpha$, whose complement has codimension greater or equal than 2, and an open subset $V$ of $P^\alpha$ on which the $\C^*$-equivariant birational map $\Phi_{\mid V}$ is an isomorphism. 
	
We know that the $H''$-action on $X^\alpha$ is of B-type by Proposition \ref{proposition:X^alphageometricrealization}; saying that it is a bordism is to say that the only $H''$-invariant divisors in $X^\alpha$ that are not extensions of divisors in $Y_-$ are the sink and the source of the action. Let $D$ be an $H''$-invariant prime  divisor on $X^\alpha$ that is not an extension of a divisor in $Y_-$. Since its intersection with $U$ is nonempty, we may consider its strict transform $\ol{D}$ into $P^\alpha$, that will be an $H''$-invariant divisor. Then $\ol{D}$ coincides with the sink or the source of $P^\alpha$, and this implies that $D$ is the sink or the source of $X^\alpha$.
	%Suppose then there exists a $\C^*$-invariant divisor $D$ on $X^\alpha$ different from $Y_{\pm}$. Then its strict transform in $P^\alpha$ is $\C^*$-invariant, thus coincides with $Y_-$ of $Y_+$, and we conclude.
\end{proof}

This concludes the proof of Theorem \ref{thm:MDR2}.

\begin{corollary}
	The $H''$-action on $X^\alpha$ is equalized at the sink $Y_-$ and the source $Y_+$. 
\end{corollary}

\begin{proof}
	It suffices to notice that thanks to Remark \ref{rem:equivariant} the $\C^*$-action on $P^\alpha$ is equalized at the sink and the source and the birational map $\Phi:P^\alpha\dashrightarrow X^\alpha$ is $\C^*$-equivariant.
\end{proof}

\begin{remark}\label{remark:birational1ps}
	Our construction of a geometric realization depends on the choice of a 1-parameter subgroup $\alpha\in\Mo(H)^\vee$. Given another $\beta\in \Mo(H)^\vee$ as above, the geometric realizations $X^{\alpha}$ and $X^{\beta}$ are birational. Indeed it suffices to notice that the $\P^1$-bundles $P^{\alpha}, P^{\beta}$ are $\C^*$-equivariantly birationally equivalent. However, the geometric quotients of $X^\alpha$ are independent of the choice of $\alpha$.
\end{remark}

%% file: bordismtoMDR.tex
%!TEX root = BRUS.tex

\section{From bordisms to MDBs and the induced chamber decomposition}\label{sec:bordismtoMDR}

\begin{setup}\label{setup:bordismtoMDR} 
	Let $(X,L)$ be a polarized pair endowed with a faithful $\C^*$-action, where $X$ is a normal and $\Q$-factorial projective variety, and $L$ is an ample Cartier divisor. Suppose that the $\C^*$-action is a bordism and it is equalized at the sink and the source. As in Section \ref{ssec:BB}, the critical values of the action will be denoted by $a_0,\dots,a_r$ and, without loss of generality, we assume that $\mu_L(Y_-)=0$. We set $\delta:=a_r=\mu_L(Y_+)$. %\lstodo{Equalized at sink and source}
\end{setup}

Under Set-up \ref{setup:bordismtoMDR}, in this section we construct an MDB upon the sink and the source. Moreover we prove the existence of a chamber decomposition of such an MDB, where every chamber model is the geometric quotient of the $\C^*$-action on $(X,L)$. Note that bordisms can be easily constructed via the pruning of a  polarized pair $(Z,E)$ with a $\C^*$-action satisfying some mild conditions (see Section \ref{sec:constructionbordism}). 

\begin{lemma}\label{lemma:sections}
	In the situation of Set-up \ref{setup:bordismtoMDR}, let $\tau_{\pm}\in \Q$ be two rational numbers such that $0\leq \tau_- \leq \tau_+ \leq \delta$, and %a positive integer 
	$m\in\Z_{>0}$ such that $m\tau_\pm\in\Z$. It holds that:
	$$\bigoplus_{k=m\tau_-}^{m\tau_+} \HH^0(X,mL)_k= \HH^0(X,mL-m\tau_-Y_--(m\delta-m\tau_+)Y_+).$$
\end{lemma}

\begin{proof}
Let us denote $W:=\HH^0(X,mL-m\tau_-Y_--(m\delta-m\tau_+)Y_+)\subset \HH^0(X,mL)$. 
Note first that $W$ is $\C^*$-invariant, therefore, $W=\bigoplus_k (\HH^0(X,mL)_k\cap W)$.

We will use \cite[Corollary~2.4]{WORS5} (which follows from \cite[Lemma 2.17]{BWW}), which determines the multiplicity of the $\C^*$-invariant sections of $\HH^0(X,mL)$ at the extremal fixed point components of the action. We note first that the proof of this result requires only the smoothness of the variety at the general points of $Y_\pm$, and this condition holds in our situation, because $X$ is normal and the action is of B-type. The quoted Corollary tells us that a nonzero section $s\in \HH^0(X,mL)_u$  vanishes with multiplicity precisely equal to $u$ at $Y_-$ and $m\delta-u$ at $Y_+$. This implies that $\HH^0(X,mL)_u\subset W$ if $u\in [m\tau_-,m\tau_+]$ and $\HH^0(X,mL)_u\cap W=0$ if $u\notin [m\tau_-,m\tau_+]$, and the claimed equality follows.

%
%
%tells us that a non zero section of $\HH^0(X,mL)_k$ vanishes with multiplicity precisely
%
%Since $X$ is normal and the action is of B-type, $X$ is smooth at the general points of $Y_{\pm}$. Let $y_-$ be the general point of $Y_-$, and consider the orbit having $y_-$ as sink; since $y_-$ is general, the orbit will have the source $y_+$ in $Y_+$. 
%
%Consider a smooth $\C^*$-invariant open subset $U$ of $X$ containing $y_{\pm}$. By assumption, the $\C^*$-action is equalized on $U\cap Y_{\pm}$. We then apply Corollary \cite[Corollary 2.4]{WORS5} (see also \cite[Lemma 2.17]{BWW}) to $U$, to conclude that, given a positive integer $k=m\tau_-,\ldots,m\tau_+$, a section of $\HH^0(X,mL)_k$ vanishes with multiplicity at least $k$ along $Y_-$ and at least $m\delta-k$ along $Y_+$. 
%
%For the converse note that, since the subspace $$W:=\HH^0(X,mL-m\tau_-Y_--(m\delta-m\tau_+)Y_+)\subset \HH^0(X,mL)$$ is $\C^*$-invariant, we have that $W=\bigoplus_k (\HH^0(X,mL)_k\cap W)$, Then the result follows again by \cite[Lemma 2.17]{BWW}.
\end{proof}

%We will show in Lemma \ref{lemma:sectionsink} that there exists an isomorphism between global sections of $L$ on $X$ of weight $s$, for $s=0,\ldots,m\delta$, and global sections of $L_{|Y_-}-sY_{-|Y_-}$ on $Y_-$. To this end, we first prove the following:
We will show in Lemma \ref{lemma:sectionsink} that, for any positive integer $m$ greater or equal than a certain $m_-\in \Z_{>0}$, there exists an isomorphism between global sections of $mL$ on $X$ of weight $c$, with $c\in [0,\ldots,m\delta]\cap \Z$, and global sections of $mL_{|Y_-}-cY_{-|Y_-}$ on $Y_-$. To this end, we first prove the following:

\begin{lemma}\label{lemma:qfactorialquotients}
	In the situation of Set-up \ref{setup:bordismtoMDR}, the sink $Y_-$ and the source $Y_+$ are $\Q$-factorial.
\end{lemma}

\begin{proof}
	We prove the result for $Y_-$; a similar proof works in the case of $Y_+$. Let $D$ be a prime divisor in $Y_-$, and consider its extension $e_-(D)\in \Div(X)$. By definition it is the closure $\overline{\pi^{-1}(D)}\subset X$, where $\pi:X^s(0,1)\to Y_-$. The fact that the $\C^*$-action is of B-type implies that $e_-(D)$ can also be written as $\overline{\pi'^{-1}(D)}$ where $\pi':X^-(Y_-)\to Y_-$. Then it follows that $e_-(D)\cap Y_-=D$. Since $X$ is $\Q$-factorial, there exists 	$m\in\Z_{>0}$ such that $me_-(D)=e_-(mD)$ is Cartier, and so $mD=me_-(D)\cap Y_-$ is Cartier, as well. 
\end{proof}

\begin{lemma}\label{lemma:sectionsink}
	In the situation of Set-up \ref{setup:bordismtoMDR}, there exists a positive integer $m_-$ such that for $m\geq m_-$ and every $c\in [0,\ldots,m\delta]\cap \Z$ it holds that:
	$$\HH^0(X,mL)_c\simeq \HH^0(Y_-,mL_{|Y_-}-cY_{-|Y_-}).$$
\end{lemma}

\begin{proof}
Let us first note that, by Lemma \ref{lemma:sections} it follows that, for every $m\in \Z_{>0}$, and every $c\in \Z_{\ge 0}$, $c\leq m\delta$, we have a commutative diagram with exact columns:
	\begin{center}
		\begin{tikzcd}
			0 \arrow[d]                                        & 0 \arrow[d]                      \\
			{\bigoplus\limits_{k=c+1}^{m\delta}\HH^0(X,mL)_k} \arrow[d] \arrow[r, "\simeq"] & {\HH^0(X,mL-(c+1)Y_-)} \arrow[d]  \\
			{\bigoplus\limits_{k=c}^{m\delta}\HH^0(X,mL)_k} \arrow[d] \arrow[r, "\simeq"] & {\HH^0(X,mL-cY_-)} \arrow[d]      \\
			{\HH^0(X,mL)_c} \arrow[d] \arrow[r]  & {\HH^0(Y_-,mL_{|Y_-}-cY_{-|Y_-})} \\
			0                                                  &                        &                                 
		\end{tikzcd}
	\end{center}
	It is then enough to show that there exists $m_-$ such that for every $m\geq m_-$, and every $c=0,\ldots,m\delta$ the restriction map $\HH^0(X,mL-cY_-)\to \HH^0(Y_-,(mL-cY_-)_{|Y_-})$
	is surjective, or, equivalently, that the rational map $$|mL-cY_-|\dashrightarrow |(mL-cY_-)_{|Y_-}|$$ is surjective.

We start by claiming that there exists $m_-$ such that for every $m\geq m_-$, $\HH^0(X,mL)_c\neq 0$ for every $c \in [0,\ldots,m\delta]\cap \Z$. In fact, let $C$ be the closure of the general $\C^*$-orbit in $X$, which has extremal fixed points in $Y_-$, $Y_+$, respectively. The generality assumption implies that the $\C^*$-action on $C$ is faithful, that its extremal points are smooth points of $Y_\pm$ and, by the Bia{\l}ynicki-Birula decomposition, that $C$ is isomorphic to $\P^1$. By Serre vanishing, there exists an integer $m_-$ such that for every $m\geq m_-$ the restriction map:
$$
\HH^0(X,mL)\to \HH^0(C,mL_{|C})
$$
is surjective and $\C^*$-equivariant. By \cite[Corollary~3.2]{RW} one has $\HH^0(C,mL_{|C})\simeq \HH^0(\P^1,\cO_{\P^1}(m\delta))$, and since the set of weights of the induced $\C^*$-action on the vector space $\HH^0(\P^1,\cO_{\P^1}(m\delta))$ is $[0,m\delta]\cap \Z$, the claim follows. %it follows that $\HH^0(X,mL)_k\neq 0$ for every $k in [0,\ldots,m\delta]\cap \Z$.

Putting it together with the commutative diagram above, the claim implies that the map $\HH^0(X,mL-(c+1)Y_-)\to \HH^0(X,mL-cY_-)$ is not surjective for $m\geq m_-$ and every $c\in [0,m\delta]\cap \Z$, that is, we have a strict inclusion:
\begin{equation}\label{eq:linsyst}
|mL-(c+1)Y_-|+Y_-\subsetneq |mL-cY_-|.
\end{equation}
In particular (since the projective space $|mL-cY_-|$ is spanned by $\C^*$-invariant elements), there exists a $\C^*$-invariant effective divisor $D_1\in|mL-cY_-|$ whose support does not contain $Y_-$. Using Lemma \ref{lemma:extensioninvariantdivisors} It follows that $$D_1=e_0({D'_1})+aY_+\mbox{ for some $a\geq 0$ and some $D'_1\in \Div(Y_-)$;}$$  here we are denoting by  $e_0:\Div(Y_-)=\Div(\GX(0,1))\to\Div(X)$ the extension map of divisors introduced in Definition \ref{definition:extension}. By restricting the above equality to $Y_-$ we get that $D_{1\mid Y_-}=D_1'$, hence
$$D_1=e_0({D_1}_{|Y_-})+aY_+\mbox{ for some $a\geq 0$}.$$

Let us now conclude the proof of the statement by showing that the restriction map $|mL-cY_-|\dashrightarrow |(mL-cY_-)_{|Y_-}|$ is surjective. Given $D'\in |(mL-cY_-)_{|Y_-}|$, Lemma \ref{lemma:extensionlinearequivalence} tells us that $e_0(D')\equiv e_0({D_1}_{|Y_-})$ which we have proven to be linearly equivalent to $D_1-aY_+$, for some $a\geq 0$. Then it follows that $(e_0(D')+aY_+)_{|Y_-}=D'$, and $e_0(D')+aY_+\in |D|$.

%Now we claim that, for $m\geq m_-$ and every $k\in [0,m\delta]\cap \Z$, there exists $D_1\in|mL-cY_-|$ such that 
%$$e_0(D_1|_{Y_-})\equiv D_1-aY+\mbox{ for some }a\geq 0;$$
% here we are denoting by  $e_0:\Div(Y_-)=\Div(\GX(0,1))\to\Div(X)$ the extension map of divisors introduced in Definition \ref{definition:extension}. 

\end{proof}

\begin{remark}\label{remark:sectionsource}
%In a similar way, one can show that for any $c=0,\ldots, m\delta$ it holds:
%	$$\HH^0(X,mL)_c\simeq \HH^0(Y_+,mL_{|Y_+}-(m\delta-c)Y_{+|Y_+}).$$
Following a similar proof, one may show that, in the situation of Set-up \ref{setup:bordismtoMDR}, there exists a positive integer $m_+$ such that for any $m\geq m_+$, and for any $c\in [0,\ldots,m\delta]$, it holds that
$$\HH^0(X,mL)_c\simeq \HH^0(Y_+,mL_{|Y_+}-(m\delta-c)Y_{+|Y_+}).$$
\end{remark}

Notice that, since the $\C^*$-action a bordism and since $Y_{\pm}$ are $\Q$-factorial by Lemma \ref{lemma:qfactorialquotients}, the natural birational map $\psi: Y_-\dashrightarrow Y_+$ is a SQM.

\begin{theorem}\label{theorem:mdr}
	In the situation of Set-up \ref{setup:bordismtoMDR}, consider $L_-=L_{\mid Y_-}$ and let $L_+=L_{\mid Y_-}- \delta Y_{- \mid Y_-}$.	
	Then the pair $(L_-,L_+)$ is an MDB, that is the cone $\cC=\langle L_-,L_+ \rangle$ is a Mori dream region.
\end{theorem}

%\begin{theorem}\label{theorem:mdr}
%	In the situation of Set-up \ref{setup:bordismtoMDR}, the induced birational map $\psi:Y_-\dashrightarrow Y_+$ is a small $\Q$-factorial modification, and $(L_-,L_+)$, where $L_-=L_{|Y_-}$ and $L_+$ is the strict transform into $Y_-$ of $L_{|Y_+}$ is an MDB whose natural associated map is $\psi$.	
%\end{theorem}

\begin{proof}
	We show each condition of Definition \ref{def:MDRtype} is satisfied. Notice that $L_{\pm}$ are effective, and $L_-$ is ample. Using Lemma \ref{lemma:sectionsink} and Remark \ref{remark:sectionsource}, there exists a positive integer $m_0\geq m_{\pm}$ such that, for any $m\geq m_0$ and any $c\in [0,\ldots,m\delta]\cap \Z$, it holds that
	$$\HH^0(Y_-,mL_{\mid Y_-}-cY_{-\mid Y_-})\simeq \HH^0(X,mL)_c \simeq \HH^0(Y_+,mL_{\mid Y_+}-(m\delta-c)Y_{+\mid Y_+}).$$
	Let $d\geq m_0$ be a positive integer, and consider the $d$-Veronese algebras $R(Y_{\pm},L_{\mid Y_{\pm}})^{(d)}$, which are still finitely generated. Using the above identity, it is readily seen that $Y_+=\Proj R(Y_+, L_{|Y_+})^{(d)}\simeq \Proj R(Y_-,L_+)^{(d)}$. 
%	Indeed for any positive integer $m$, by setting $c=md\delta$ in the previous identity we obtain that
%	\begin{equation*}
%	\begin{split}
%	\HH^0(Y_-,mdL_+)&=\HH^0(Y_-,mdL_{\mid Y_-}-md\delta Y_{-\mid Y_-})=\\
%	&=\HH^0(X,mdL)_{md\delta}=\\
%	&=\HH^0(Y_+,mdL_{\mid Y_+}-(md\delta-md\delta Y_{+\mid Y_-}))=\\
%	&=\HH^0(Y_+,mdL_{\mid Y_+}).
%	\end{split}
%	\end{equation*}
	It remains to show that $R(Y_-;L_-,L_+)$ is finitely generated. Consider the $d$-Veronese algebra $R(X;L)^{(d)}$, which is still finitely generated being $L$ ample and $X$ projective.
	By using Lemma \ref{lemma:sectionsink}, we know that
	$$\bigoplus_{m\geq 0}\bigoplus_{k=0}^{md\delta} \HH^0(X,mL)_k\simeq \bigoplus_{m\geq 0}\bigoplus_{k=0}^{md\delta} \HH^0(Y_-,mL_{|Y_-}-kY_{-|Y_-})$$
	is finitely generated.
	Notice that we may rewrite the RHS using $L_{\pm}$, thus obtaining that
	$$R(X;L)^{(d)}\simeq \bigoplus_{(a,b)\in S} \HH^0(Y_-,aL_{|Y_-}+b(L_{|Y_-}-\delta Y_{-|Y_-})),$$
	where $S$ denotes the monoid $\frac{1}{d\delta}(\Z_{\geq 0})^{\oplus 2}\subset \Q^{\oplus 2}$.  
	We may represent this situation by means of the following image, where the black dots belong to $S$, and the empty ones to $\N^{\oplus 2}\subset S$: 
	\begin{center}
		\includegraphics[scale=0.2]{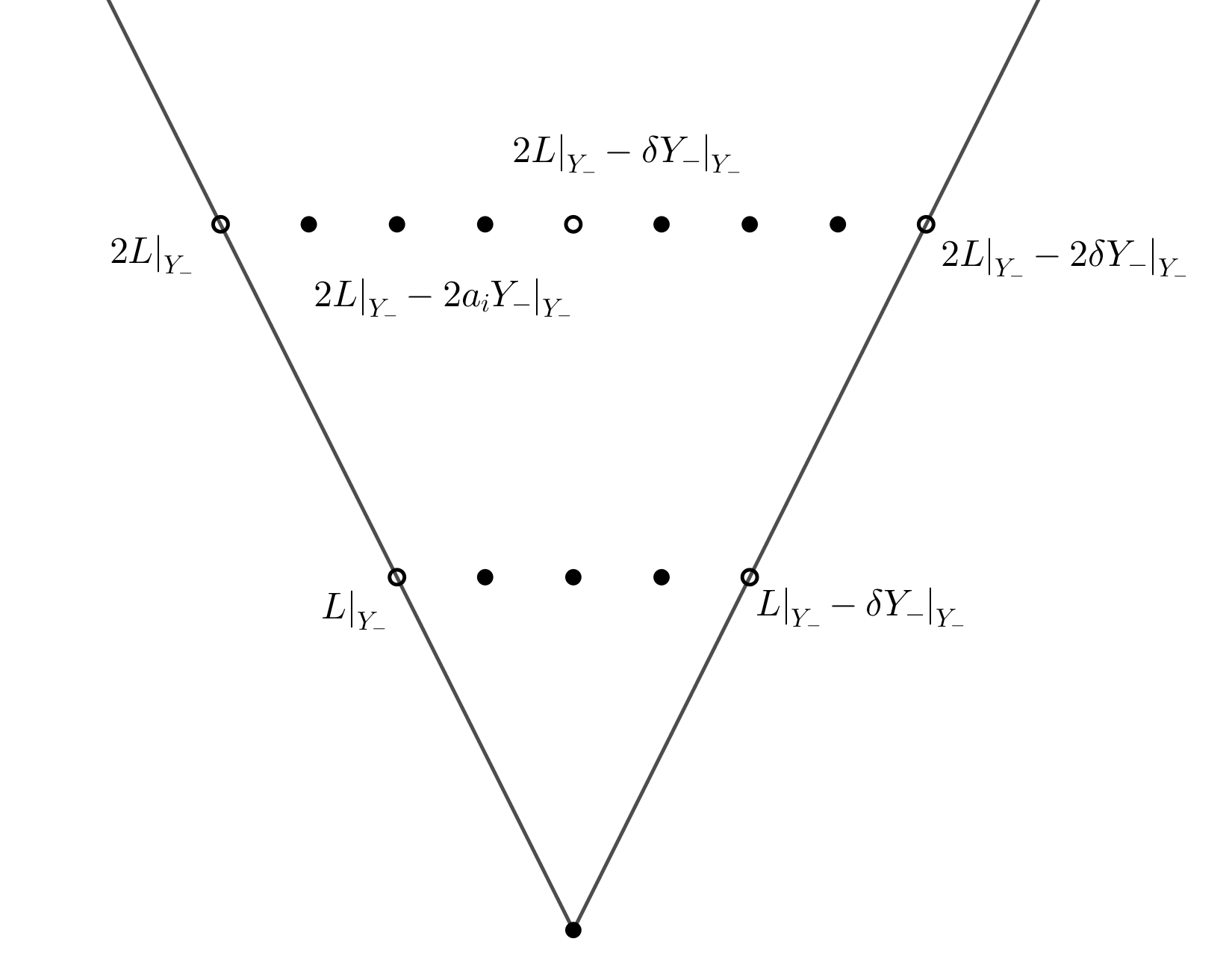}
	\end{center}
	Therefore, using \cite[Lemma 2.25]{cascinilazic} (see also \cite[Propositions 1.2.2, 1.2.4]{adhl}), we conclude that the algebra
	$$R(Y_-;L_-,L_+)=\bigoplus_{a,b\in \N^{\oplus 2}} \HH^0(Y_-,aL_{|Y_-}+b(L_{|Y_-}-\delta Y_{-|Y_-}))$$
	associated to the cone $\cC=\langle L_-,L_+ \rangle$ is finitely generated. 
\end{proof}

We conclude this section by showing that the MDB obtained in Theorem \ref{theorem:mdr} admits a chamber decomposition, which is induced by the $\C^*$-action on the polarized pair $(X,L)$. The decomposition of Mori dream regions, which reproduces the behaviour of Mori dream spaces, has been stated by \cite[Definition 2.12]{HuKeel}: we refer to \cite[Theorem 4.3]{KKL} and \cite[Proposition 9.6]{okawa} for the precise statements.

\begin{definition}\label{definition:chamber}
	Let $X$ be a normal projective variety, and let $\cC$ be a rational polyhedral cone in $\CDiv(X)_{\Q}$. We say that $\cC$ is a chamber if, for any $D_1,D_2\in \cC$ with finitely generated section ring, it holds that $\Proj R(X;D_1)\simeq \Proj R(X;D_2)$. We call the variety $\Proj R(X;D_1)$ the \emph{chamber model} of $\cC$.
\end{definition}

%\begin{definition}\label{def:invariantchamber}
%	Let $C$ be a rational polyhedral cone in $\CDiv(X)_{\Q}$. We say that $C$ is a $\C^*$-\emph{invariant chamber} if for every $D_1,D_2$ in $C^\circ\cap \CDiv(X)$ with finitely generated section ring, it holds $\Proj R(Y_-;D_1)\simeq \Proj R(Y_-;D_2)$. We call the variety $\Proj R(Y_-;D_1)$ the \emph{chamber model}.
%\end{definition}    

\begin{theorem}\label{theorem:decomposition}
	In the situation of Set-up \ref{setup:bordismtoMDR}, the cone $\cC=\langle L_-,L_+\rangle$ of Theorem \ref{theorem:mdr} admits a subdivision $$\cC=\bigcup_{i=0}^{r-1} \cC_i, \quad \cC_i=\langle L_{|Y_-}-a_iY_{-|Y_-},L_{|Y_-}-a_{i+1}Y_{-|Y_-} \rangle.$$
	Moreover, for every $i=0,\ldots,r-1$ the cone $\cC_i$ is a chamber whose model is $\GX(i,i+1)$.
\end{theorem}

%\begin{theorem}\label{theorem:decomposition}
%	In the situation of Set-up \ref{setup:bordismtoMDR}, the cone $\cC=\langle L_-,L_+ \rangle$ of Theorem \ref{theorem:mdr}  admits a subdivision $\cC=\bigcup_{i=0}^{r-1}\cC_i$, where every $\cC_i=\langle L_{|Y_-}-a_iY_{-|Y_-},L_{|Y_-}-a_{i+1}Y_{-|Y_-} \rangle$ is a $\C^*$-invariant chamber. Moreover, for every $i$ the chamber model of $\cC_i$ is $\GX(i,i+1)$.
%\end{theorem}

\begin{proof}
	The existence of such a subdivision follows immediately by recalling that $a_i<a_{i+1}$ for every $i=0,\ldots,r-1$. In order to conclude, it suffices to show that, for every $i=0,\ldots,r-1$, the cone $\cC_i$ is a chamber. Let $D=\beta(L_{|Y_-}-a_iY_{-|Y_-})+\gamma(L_{|Y_-}-a_{i+1}Y_{-|Y_-})$ be a divisor in $\cC_i$, where $\beta,\gamma\in \Q_{>0}$. Let $q$ be a positive integer such that $q\beta, q\gamma\in \N$ and $q\geq m_-$, with $m_-$ as in Lemma \ref{lemma:sectionsink}. Using again Lemma \ref{lemma:sectionsink} we obtain that
	$$\HH^0(Y_-,qD)\simeq \HH^0(X,q(\beta+\gamma)L)_{q(\beta a_i+\gamma a_{i+1})},$$
	and since $q\beta a_i+q\gamma a_{i+1}\in (q(\beta+\gamma)a_i,q(\beta+\gamma)a_{i+1})$, using Remark \ref{remark:projquotients} and the above isomorphism it holds that $$\Proj R(Y_-;qD)\simeq \Proj \bigoplus_{m\geq 0 }\HH^0(X,mq(\beta+\gamma)L)_{mq(\beta a_i+\gamma a_{i+1})} \simeq  \GX(i,i+1).$$ 
\end{proof}

%\begin{proof}
%	It is enough to show that, for every $i=0,\ldots,r-1$, the cone $\cC_i$ is a $\C^*$-invariant chamber. To this end, we prove that the chamber model of every $\cC_i$ is $\GX(i,i+1)$. Let $D=\beta(L_{|Y_-}-a_iY_{-|Y_-})+\gamma(L_{|Y_-}-a_{i+1}Y_{-|Y_-})$ be a divisor in $\cC_i$, where $\beta,\gamma\in \Q_{>0}$. Up to considering a multiple, we may assume that $\beta+\gamma\in \N$. Using Lemma \ref{lemma:sectionsink} we obtain that
%	$$\HH^0(Y_-,D)\simeq \HH^0(X,(\beta+\gamma)L)_{\beta a_i+\gamma a_{i+1}},$$
%	and since $\beta a_i+\gamma a_{i+1}\in ((\beta+\gamma)a_i,(\beta+\gamma)a_{i+1})$, using Remark \ref{remark:projquotients} and the above isomorphism it holds that $$\Proj R(Y_-;D)\simeq \Proj \bigoplus_{m\geq 0 }\HH^0(X,m(\beta+\gamma)L)_{m(\beta a_i+\gamma a_{i+1})} \simeq  \GX(i,i+1).$$ 
%\end{proof}

%% file: example.tex
%!TEX root = BRUS.tex

\section{An example: factorization of the cubo-cubic standard Cremona transformation of $\P^3$}\label{sec:Cremona}

In this section we will illustrate the concepts of pruning and of the geometric realization of an MDB in an example coming from classical algebraic geometry, namely the cubo-cubic standard Cremona involution of $\P^3$. To do so, we will use the language of toric varieties, since the algebraic approach presented in these constructions  can be easily implemented in the toric setting. This approach, which was already used in \cite[Example 6.1]{WORS6}, will be studied in more detail in \cite{BOS}.

%In thiis section we use the factorization of the cubo-cubic standard Cremona involution of $\P^3$ to provide explicit examples both for the construction of the pruning (cf. \S \ref{sec:constructionbordism}) and for the construction of a geometric realization of an MDB (cf. Theorem \ref{theorem:mdr}). To do so, we will use toric varieties, since the algebraic approach presented in the construction of geometric realization (cf. \S \ref{sec:MDR2}) can be easily implemented in the toric setting. The latter will become part of \cite{BOS}.

Consider the cubo-cubic standard Cremona involution of $\P^3$:

\begin{equation*}
	\begin{aligned}
		T: \P^3 & \dashrightarrow \P^3 \\
		(x:y:z:w) & \mapsto (x^{-1}:y^{-1}:z^{-1}:w^{-1})=(yzw:xzw:xyw:xyz).
	\end{aligned}
\end{equation*}

Let $W\to \P^3$ be the blow-up of $\P^3$ along the $4$ coordinates points $P_0=(1:0:0:0),\ldots,P_3=(0:0:0:1)$, and let $E_i$, with $i=0,\ldots,3$, be the associated exceptional divisors. Then the map $T$ is given by the linear system of cubics which are singular along $P_0,\ldots,P_3$, that is $|3H-2E_0-2E_1-2E_2-2E_3|$ (cf. \cite[\S 2]{LafaceUgagliaP3}), where $H$ denotes the pullback into $W$  of the hyperplane section in $\P^3$.

As explained in \cite[Example 5.1.6]{DolgachevCremona}, the birational map $T$ admits a factorization
\begin{equation*}\label{eq:Cremonafact} 
\xymatrix@C=15mm{\P^3\ar@{-->}[r]^{\psi_1}&W \ar@{-->}[r]^{\psi_2} &\widetilde{W} \ar[r]^{\psi_3} &\P^3}
\end{equation*}
where 
\begin{itemize}
	\item $\psi_1: \P^3 \dashrightarrow W$ is the inverse of the blow-up of $\P^3$ along the four coordinate points $P_0,P_1,P_2,P_3$; 
	\item $\psi_2 : W\dashrightarrow \widetilde{W}$ is the  flip of the strict transform of the six lines joining the coordinate points;
	\item $\psi_3: \widetilde{W} \to \P^3$ is the blow-up of $\P^3$ along the four coordinate points.
\end{itemize}
Since all these varieties are toric, we may have a better understanding of the situation by looking at the associated moment polytopes (Figure \ref{fig:cremonafact}).

\begin{figure}[htp]
	\centering
	\hfill
	\includegraphics[scale=0.06]{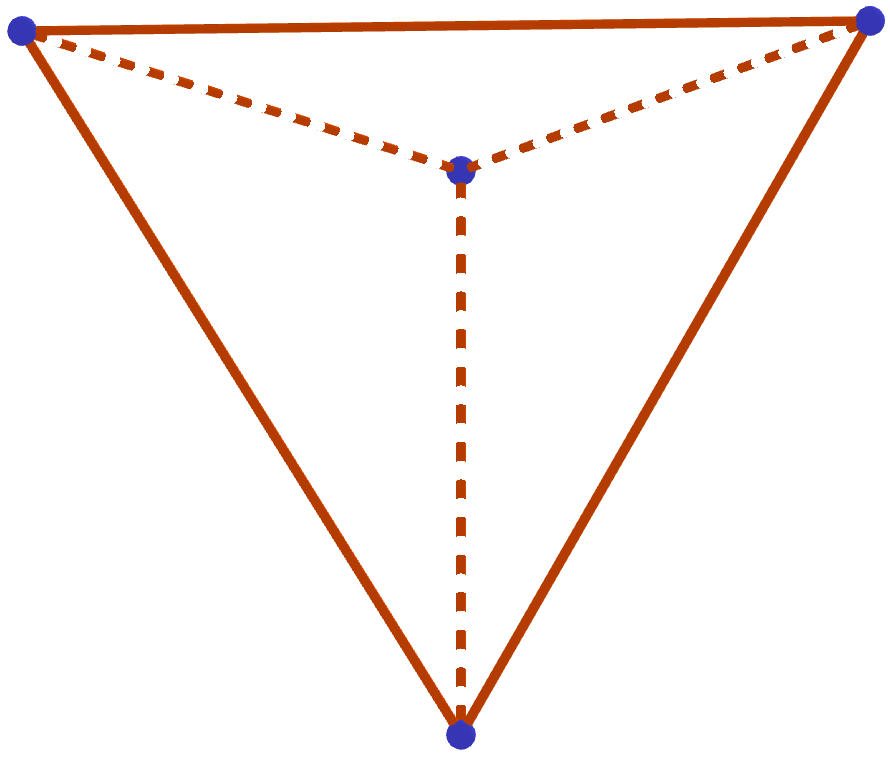}\hfill %\hfill
	\includegraphics[scale=0.06]{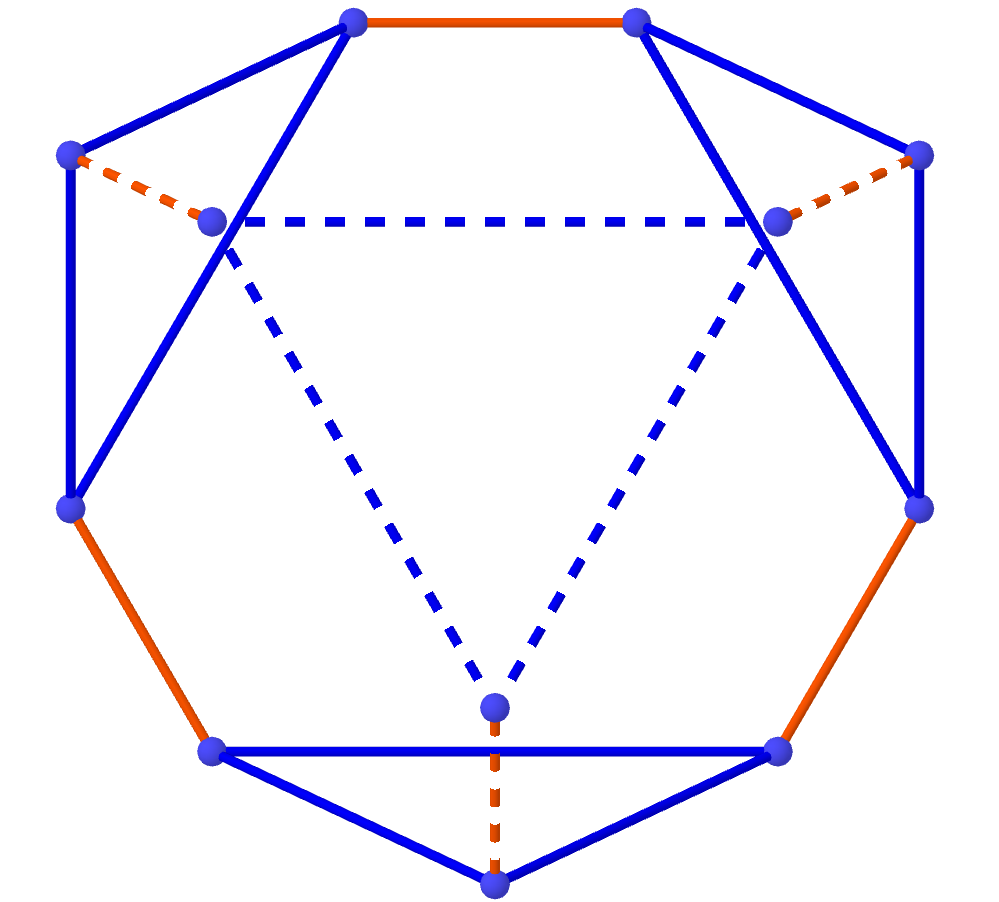}\hfill %\hfill
	\includegraphics[scale=0.06]{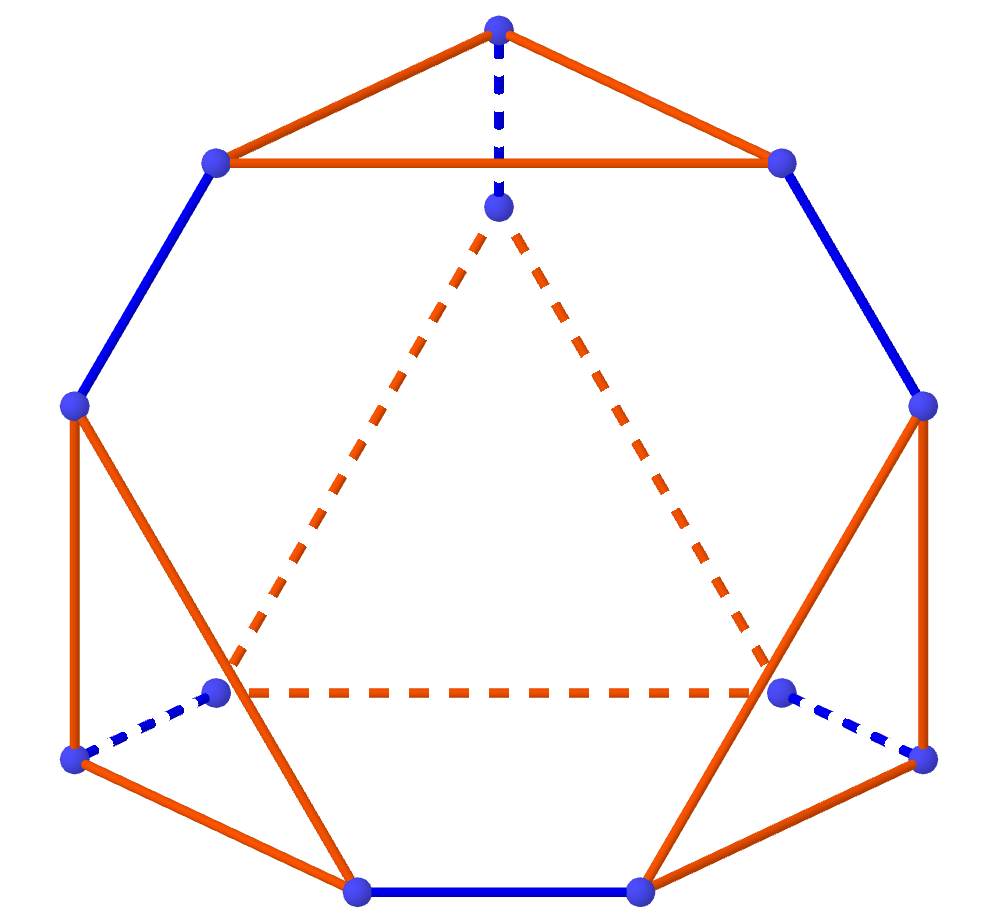}\hfill %\hfill
	\includegraphics[scale=0.06]{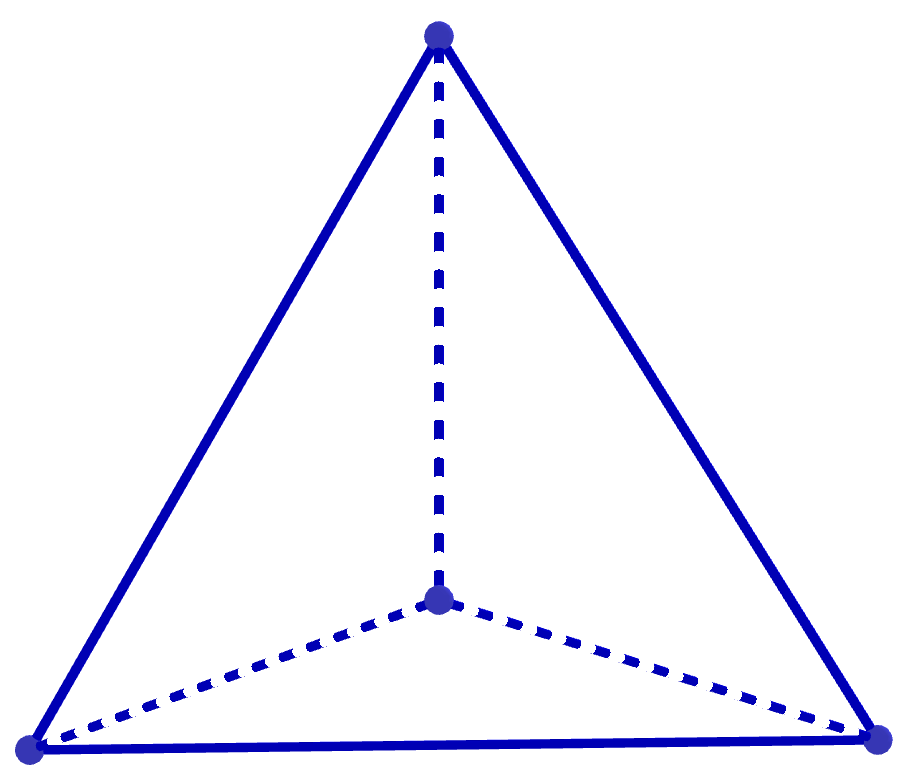}\hfill %\hfill
	\hfill
	\caption{\label{fig:cremonafact}The polytopes  associated respectively to $\P^3,W,\widetilde{W}$ and $\P^3$.} %width=.3\textwidth
%	\label{figure:Toric3}
\end{figure}

%A natural geometric realization of the map $T$ is the diagonal action of $\C^*$ on $X=\P^1\times \P^1\times \P^1\times \P^1$, defined as:
A natural geometric realization of the map $T$ is given by $X=\P^1\times \P^1\times \P^1\times \P^1$, together with the diagonal action of $\C^*$ defined as:
$$t\big((x_0\!:\! x_1),(y_0\!:\!y_1),(z_0\!:\!z_1),(w_0\!:\!w_1)\big)=\big((x_0\!:\!tx_1),(y_0\!:\!ty_1),(z_0\!:\!tz_1),(w_0\!:\!tw_1)\big),$$
whose sink and source are respectively $y_-=(\infty,\infty,\infty,\infty)$ and $y_+=(0,0,0,0)$, where we set $\infty:=(0:1)\in \P^1$, $0:=(1:0)\in \P^1$. In fact $T$ is precisely the induced birational map from $\P(T_{X,y_-})\dashrightarrow \P(T_{X,y_+})$ that sends a (general) tangent direction $x_-\in \P(T_{X,y_-})$ to $x_+\in \P(T_{X,y_+})$ if there exists an orbit with sink and source at $y_\pm$, and having tangent directions $x_\pm$ at $y_\pm$. 

Different polarizations of $X$ will give rise to different possible prunings of the $\C^*$-action and, subsequently, will produce different factorizations of the natural birational map $\psi$. In this case, if we choose the polarization given by the ample line bundle $L:=\cO_{X}(1,1,1,1)$, then the action on $(X,L)$ has bandwidth and criticality equal to four and, without loss of generality, we may assume that the weights with respect to $L$ on the fixed point components of the actions are $0,1,2,3,4$. Moreover, it has four geometric quotients 
$$\GX(0,1)\simeq \P^3,\quad \GX(1,2)\simeq W,\quad \GX(2,3)\simeq \widetilde{W},\quad \GX(3,4)\simeq \P^3,$$ which are the toric varieties described above. We represent this situation in the following picture, where between every two consecutive weights  $i,i+1$ we have written the corresponding geometric quotient $\GX(i,i+1)$, for $i=0,\ldots,3$.

\begin{figure}[h!]
	\centering
	\includegraphics[scale=0.5]{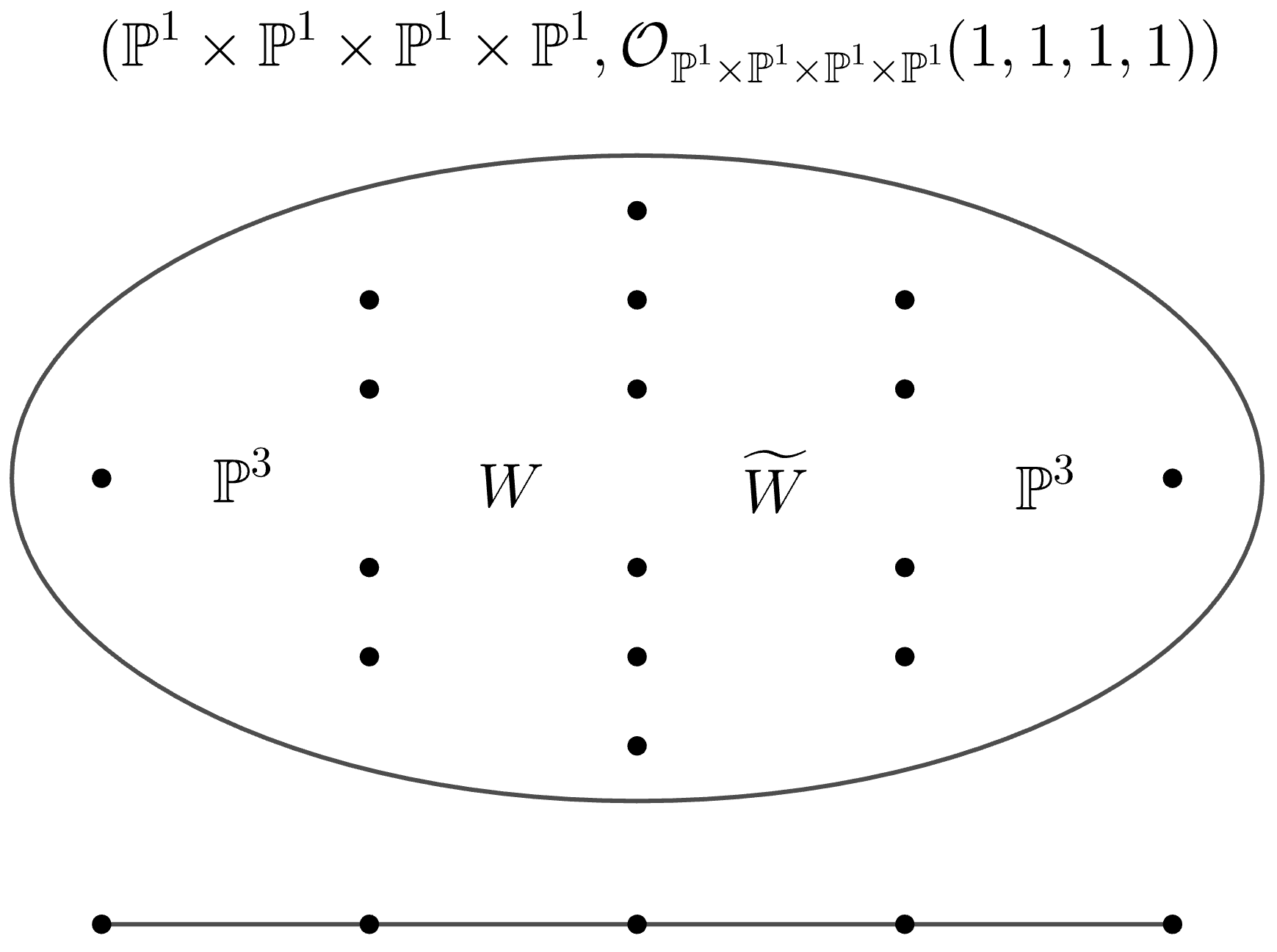}\hfill %\hfill
\end{figure}

We may then consider different prunings of the $\C^*$-action on $(X,L)$; for instance:

\begin{itemize}
\item the induced $\C^*$-action on the pruning $X_1$ of $X$ with respect to $\rho_-=1/2$, $\rho_+=7/2$ is of B-type, with sink and source isomorphic to $\P^3$; it has criticality $4$ and it is not a bordism;
\item the induced $\C^*$-action on the pruning $X_2$ of $X$ with respect to $\rho_-=3/2$, $\rho_+=5/2$ is a bordism of criticality $2$, with sink and source isomorphic to $W$ and $\widetilde{W}$, respectively; the associated birational map is the SQM $\psi_2$;
\item the prunings of $X$ satisfying that $[\rho_-,\rho_+]\cap\Z=\emptyset$ are decomposable $\P^1$-bundles over one of the geometric quotients of $X$; for instance, the pruning with respect to $\rho_-=3/2$, $\rho_+=5/3$ is a decomposable $\P^1$-bundle over $W$.
\end{itemize}

Let us also illustrate, in the context of our example, the construction of bordisms out of MDBs that we have discussed in Section \ref{sec:MDR2}. 

Let us consider the SQM $\psi_2:W\dashrightarrow \widetilde{W}$ defined above. The varieties $W$ and $\widetilde{W}$ are smooth 
projective toric varieties, hence they are Mori dream spaces (cf. \cite[Cor. 2.4]{HuKeel}), and in particular any choice of pair of effective Cartier divisors gives a MDB (cf. \cite[Lemma 2.7]{Castravet}). 

So we identify $\NU(W)$ with $\NU(\tl{W})$ via $\psi_2$, and start the construction by considering  two line bundles $L,\tl{L}\in \NU(W)$ which are respectively ample on $W,\tl{W}$. %two coprime positive integers $\alpha_-$, $\alpha_+$, 
%\lstodo{Shall we use $W_-$, W_+$ instead of $W,\tl{W}}
Set $P:=\P_W(L\oplus \tl{L})$: then, following the proof of Lemma \ref{lemma:PbirationaXalpha},
$$
X(L,\tl{L}):=\Proj\big(R(P;\cO_P(1))\big) 
$$ 
%\lstodo{We should check this (which is the construction we presented in the paper) with {\tt sage}; in this way we also avoid discusing unprunings}
is a geometric realization of the map $\psi_2$. In this way we may construct many geometric realizations of $\psi_2$, different, in general, of the variety $X_2$ defined previously; we obtain $X_2$ by choosing:
$$
L=\cO_W(5H-2E_0-2E_1-2E_2-2E_3),\quad\tl{L}=\cO_W(7H-4E_0-4E_1-4E_2-4E_3).
$$

Finally, we note that, in the context of our example, the fact that $W$, $\tl{W}$ and the map $\psi_2$ are toric allows us to use techniques of toric geometry to construct $P$ and, subsequently, $X(L,\tl{L})$; explicit details on how these toric constructions work will be presented in the forthcoming paper \cite{BOS}. By means of the mathematical software {\tt SageMath} (cf. \cite{sagemath}), we have constructed a moment polytope for $X(L,\tl{L})$; it is the $4$-dimensional polytope in $\R^4$ whose vertices are determined by the columns of the matrix below:
\setcounter{MaxMatrixCols}{30}
\[\setlength{\arraycolsep}{2.5pt}
\renewcommand\arraystretch{1}
\hfsetfillcolor{gray!10}
\hfsetbordercolor{gray}
\footnotesize{\begin{pmatrix}
		0 & 0 & 0 & 0 & 0 & 0 & 4 & 4 & 4 & 6 & 6 & 6 & 0 & 0 & 0 & 6 & 6 & 6 & 0 & 0 & 0 & 2 & 2 & 2 & 6 & 6 & 6 & 6 & 6 & 6  \\
		6 & 4 & 0 & 4 & 6 & 0 & 0 & 6 & 0 & 0 & 4 & 0 & 6 & 0 & 6 & 0 & 6 & 0 & 6 & 6 & 2 & 0 & 6 & 6 & 0& 6 & 2 & 6 & 0 & 2  \\
		0 & 0 & 6 & 6 & 4 & 4 & 0 & 0 & 6 & 0 & 0 & 4 & 0 & 6 & 6 & 6 & 0 & 0 & 6 & 2 & 6 & 6 & 6 & 0 & 6& 2 & 6 & 0 & 2 & 0  \\
		0 & 0 & 0 & 0 & 0 & 0 & 0 & 0 & 0 & 0 & 0 & 0 & 2 & 2 & 2 & 2 & 2 & 2 & 6 & 6 & 6 & 6 & 6 & 6 & 6 & 6 & 6 & 6 & 6 & 6 \\
\end{pmatrix}}\]
The facets determined by the vertices whose $4$-th coordinates are equal to $0$ and $6$ are respectively the moment polytopes of the varieties $W$, $\tl{W}$. 
Moreover, consider the $\C^*$-action  associated to the projection onto the last coordinate.  With respect to the polarization determined by the polytope, the action on $X(L,\tl{L})$ has criticality $2$, with sink and source respectively isomorphic to $W,\widetilde{W}$, and six inner fixed points all of  weight $2$, which correspond to the six lines that we flip in the SQM $\psi_2:W\dashrightarrow \tl{W}$.